\documentclass[11pt]{article}
\usepackage{amsmath,amssymb,amsthm,geometry,hyperref,lmodern}
\hypersetup{hidelinks}

\newtheorem{theorem}{Theorem}[section]
\newtheorem{lemma}[theorem]{Lemma}
\newtheorem{proposition}[theorem]{Proposition}

\newcommand{\PP}{\mathbb{P}}
\newcommand{\F}{\mathbb{F}}
\newcommand{\Span}{\operatorname{span}}
\newcommand{\Ann}{\operatorname{Ann}}
\newcommand{\Vbar}[1]{\overline{V_{#1}}}

\newcommand{\bv}{\mathbf{v}}
\newcommand{\btv}{\widetilde{\mathbf{v}}}
\newcommand{\bp}{\mathbf{p}}
\newcommand{\bw}{\mathbf{w}}
\newcommand{\bb}{\mathbf{b}}
\newcommand{\bz}{\mathbf{z}}
\newcommand{\by}{\mathbf{y}}

\newcommand{\cC}{\mathcal{C}}

\newcommand{\cA}{\mathcal{A}}
\newcommand{\cI}{\mathcal{I}}
\newcommand{\cZ}{\mathcal{Z}}
\newcommand{\cG}{\mathcal{G}}
\newcommand{\cH}{\mathcal{H}}
\newcommand{\cD}{\mathcal{D}}

\title{On the MDS conjecture in even characteristic}
\author{Bocong Chen$^1$ and Jing Huang$^2$\footnote{E-mail addresses: {\it
bocongchen@foxmail.com (B. Chen),
jhuangmath@foxmail.com (J. Huang),
}}
}

\date{\small
$1.$ School of Mathematics, South China University of Technology, Guangzhou 510641, China\\
$2.$
School of Mathematics and Information Science,
Guangzhou University, Guangzhou 510006, China\\
}

\begin{document}
\maketitle

\begin{abstract}
We prove that every nontrivial linear $[n,k]$ maximum distance separable (MDS)
code over a finite field $\F_q$ of characteristic two satisfies
$n\le q+1$, except in dimensions $k=3$ and $k=q-1$, where
$n\le q+2$. This resolves the linear MDS conjecture in even characteristic.

\vspace{0.2cm}

\end{abstract}

\section{Introduction}

Let $\F_q$ be the finite field with $q$ elements. A linear $[n,k]$
code over $\F_q$ is a $k$-dimensional subspace of $\F_q^n$.
If its minimum Hamming distance is $d$, the Singleton bound gives
$d\le n-k+1$~\cite{Singleton64}. A code attaining equality is called
\emph{maximum distance separable}, or \emph{MDS}. Equivalently,
every $k$ columns of a generator matrix are linearly independent,
so that any $k$ coordinates determine a codeword. We call the code
\emph{nontrivial} when $2\le k\le n-2$. Its dual is an
$[n,n-k]$ MDS code; see~\cite[Chapter~7]{HP03} and
\cite[Chapter~11]{MS77}. The basic extremal question is how long
such a code can be when the field and the dimension are fixed.

This question has a geometric formulation that predates the
coding-theoretic terminology. An \emph{arc} in
$\operatorname{PG}(k-1,q)=\PP(\F_q^k)$ is a set of at least $k$
points such that every $k$ of them span the whole projective space.
The projective points represented by the columns of an MDS generator
matrix form an arc, and every arc gives an MDS code. In his 1955
study of normal rational curves, Segre~\cite{Segre55a} posed three
related problems concerning the maximum size and structure of arcs,
their characterization as normal rational curves, and their
extendibility. The maximum-size question is the geometric origin of
the MDS conjecture. Segre's theorem that every $(q+1)$-arc in
$\operatorname{PG}(2,q)$ is a conic when $q$ is odd
\cite{Segre55b} provided an early instance in which the extremal
size and the geometry could both be determined.

The expected answer is suggested by two classical constructions.
Normal rational curves have $q+1$ points in the relevant dimensions
and correspond to extended Reed--Solomon codes
\cite{RS60,HP03,BL19}. When $q$ is even, a nonsingular conic together
with its nucleus, the common point of its tangent lines, is a
\emph{hyperoval}: an arc of size $q+2$ in $\operatorname{PG}(2,q)$.
It gives an $[q+2,3]$ MDS code, whose dual has dimension $q-1$.
The \emph{linear MDS conjecture} asserts that a nontrivial linear
MDS code has length at most $q+1$, with precisely these two
even-characteristic exceptions, for which the bound is $q+2$
\cite{MS77,BL19}. Thus the conjecture asks whether the lengths
supplied by the classical constructions are optimal. In coding
terms, it determines the largest block length compatible with
recovery from every pattern of $n-k$ erasures.

Much of the progress on this question grew from Segre's method of
tangents. For a planar arc, a tangent is a line meeting the arc in
exactly one point. Segre associated an algebraic curve in the dual
plane with these tangent lines, converting incidence conditions
into polynomial identities~\cite{Segre67}.
Bruen, Thas and Blokhuis~\cite{BTB88} developed the corresponding
higher-dimensional construction in even characteristic and obtained
results on the size, structure and extension of large arcs.
Storme and Thas~\cite{ST93} strengthened the resulting bounds.
Their work establishes the conjectured bound, with its stated
exceptions, for
$
 k<\frac{\sqrt q}{2}+\frac{15}{4}
$
when $q$ is even, as well as in the dimensions obtained by duality.
Separate geometric arguments settled low-dimensional cases:
Casse~\cite{Casse69} treated $k=4,5$;
Casse and Glynn~\cite{CG84} obtained the case $k=6$ through their
classification of large arcs in $\operatorname{PG}(4,q)$;
and Maruta and Kaneta~\cite{MK91} obtained $k=7$ for $q\ge16$.
Here $k$ denotes the code dimension, one more than the projective
dimension.

The tangent method also led to general results in odd
characteristic. Voloch~\cite{Voloch91} used bounds for rational
points on algebraic curves to control large planar arcs over fields
of nonsquare order. Ball and Lavrauw~\cite{BL18} later obtained
stronger conic-containment results for odd square fields by a
polynomial approach based on a scaled form of the lemma of tangents.
Projection arguments turn these planar results into MDS bounds in
ranges of dimensions growing with the field order. The exposition
of Ball and Lavrauw~\cite{BL19} brings together these geometric
methods, the relevant projection theorems, and the tangent
hypersurface construction used in the present paper.

A further development was to extract linear systems directly from
the tangent identities. If $q=p^e$ with $p$ prime,
Ball~\cite{Ball12} proved the conjecture for $k\le p$, thereby
settling it over prime fields. Ball and De Beule~\cite{BDB12}
extended the range to $k\le2p-2$ over proper extension fields.
Chowdhury~\cite{Chow16} recast these arguments in terms of
determinant-weighted inclusion matrices and related their ranks to
the extendibility of arcs. These characteristic-dependent ranges
explain an important distinction between the two parities: when
$p=2$, the bound $k\le2p-2$ reaches only dimension two.
In recent preprints, Deng, Feng and Vasil'ev~\cite{DFV26} prove the
MDS bound for $4\le k\le q-3$ under a two-transitivity hypothesis
on the permutation automorphism group, while Fan~\cite{Fan26}
reports a range of dimensions linear in $q$ for each fixed odd
characteristic, using determinant relations and inclusion matrices.

In this paper we prove the linear MDS conjecture in even
characteristic, for every field order and every nontrivial dimension.

\begin{theorem}\label{thm:even}
Let $q$ be an even prime power, and let $\cC$ be a nontrivial
linear $[n,k]$ MDS code over $\F_q$. Then
\[
 n\le
 \begin{cases}
 q+2,& k\in\{3,q-1\},\\
 q+1,& k\notin\{3,q-1\}.
 \end{cases}
\]
\end{theorem}

For $q\ge4$, extended Reed--Solomon codes, hyperoval codes and their
duals attain the bounds in every admissible dimension
$2\le k\le q-1$. Thus Theorem~\ref{thm:even} determines the exact
maximum length throughout characteristic two, without a restriction
relating the dimension to the field order. Geometrically, it settles
the maximum-size part of Segre's problem in this characteristic;
the classification of arcs attaining the maximum is a separate
question.

The proof uses the tangent-envelope polynomial in the form developed
by Ball and Lavrauw~\cite{BL19}. A differential identity and
characteristic-two interpolation produce canonical local subcodes
of the dual code from extended Reed--Solomon spaces. We estimate
the dimension of their sum by projecting the arc to a plane and
expressing the common annihilator through a colon ideal in three
variables. A Frobenius derivation gives the necessary injectivity
below the critical degree. The remaining degree is controlled by
the unimodality theorem for Artinian Gorenstein Hilbert functions
in embedding dimension at most three
\cite{Stanley78,zanello}. The resulting span is too large to lie
in the common shortened dual code. This dimension comparison
excludes length $q+2$ in the nonexceptional dimensions; duality
and puncturing complete the argument.

As this manuscript was being completed, we became aware of an independent report by Kumar~\cite{Kumar26} that also claims to
establish Theorem~\ref{thm:even}. 
His argument uses a common space
of lifts of a planar arc, whereas we construct local subcodes of the
dual code by Reed--Solomon interpolation. The arguments share a
central algebraic mechanism: both use a Frobenius obstruction to
reduce the dimension bound to a ternary colon ideal, and combine
low-degree injectivity with Gorenstein Hilbert-function unimodality.
Kumar proves the required unimodality through a rank-sequence
argument; we apply the codimension-three Gorenstein theorem directly.
The local constructions and their relation to the original code
provide the principal difference between the two presentations.

Section~2 collects the geometric, algebraic and coding-theoretic
preliminaries. Section~3 identifies the required interpolation space
with an extended Reed--Solomon code and describes its dual.
Section~4 constructs the local dual subcodes and reduces the
dimension estimate to a ternary colon ideal. Section~5 proves
the estimate and completes the proof of Theorem~\ref{thm:even}.

\section{Preliminaries}
In this section, we establish the basic notation for the paper and recall  foundational results from
 finite geometry,  commutative algebra and algebraic coding theory that underpin our argument.
Let $q=2^e$ with $e\geq3$, let $k \ge 3$ be an integer, let $V = \F_q^k$ be a $k$-dimensional vector space over $\F_q$, and let $V^*$ be its dual space. A collection of vectors $\cA = \{\bv_1, \ldots, \bv_n\} \subset V \setminus \{\mathbf{0}\}$ indexed by $\cI = \{1, \ldots, n\}$ is called a \emph{projective arc} (or an \emph{arc} in $\PP(V)$) if any $k$ vectors in $\cA$ are linearly independent over $\F_q$ (assuming $n \ge k$). Equivalently, the $\bv_i$ form the columns of a generator matrix of a linear $[n,k]$ MDS code $\cC$ over $\F_q$.

For any index subset $\Lambda \subset \cI$ with $|\Lambda| = k-2$, define the $(k-2)$-dimensional subspace and the corresponding $2$-dimensional quotient space:
\[
 U_\Lambda = \Span_{\F_q}\{\bv_i : i \in \Lambda\}, \qquad \Vbar{\Lambda} = V / U_\Lambda,
\]
together with the $2$-dimensional annihilator subspace in $V^*$:
\[
 \Ann_{V^*}(U_\Lambda) = \{\phi \in V^* : \phi(\bv_i) = 0 \text{ for every } i \in \Lambda\}.
\]
Let $\Lambda^c = \cI \setminus \Lambda$, so that $|\Lambda^c| = n - k + 2$. Define the set of missing dual points on the projective line $\PP(\Ann_{V^*}(U_\Lambda))$ as
\[
 M_\Lambda = \{ [\phi] \in \PP(\Ann_{V^*}(U_\Lambda)) : \phi(\bv_i) \ne 0 \text{ for every } i \in \Lambda^c \}.
\]
By the arc property, for each $i \in \Lambda^c$, the evaluation form $\phi \mapsto \phi(\bv_i)$ on $\Ann_{V^*}(U_\Lambda)$ is nonzero with $1$-dimensional kernel $\Ann_{V^*}(U_{\Lambda \cup \{i\}})$, and these kernels determine $n - k + 2$ distinct points on the projective line $\PP(\Ann_{V^*}(U_\Lambda))$. Since $\PP(\Ann_{V^*}(U_\Lambda))$ has $q + 1$ rational points, the complement $M_\Lambda$ consists of precisely $h = (q + 1) - (n - k + 2) = q + k - 1 - n$ points. Assuming $h \ge 1$, write $M_\Lambda = \{[\phi_{\Lambda,1}], \ldots, [\phi_{\Lambda,h}]\}$ for some fixed nonzero representatives $\phi_{\Lambda,j} \in \Ann_{V^*}(U_\Lambda)$. For each $1 \le j \le h$, let $M_{\Lambda,j} : \Ann_{V^*}(U_\Lambda) \to \F_q$ be a representative nonzero linear form whose kernel is the line $\Span_{\F_q}\{\phi_{\Lambda,j}\}$.
By \cite[Theorem 13.1]{BL19} of Ball and Lavrauw, specialized to the
even-characteristic setting, we have the following result:
\begin{lemma}\label{lem:input}
Let $q$ be an even prime power, let $k \ge 3$ and $h \ge 1$ be integers, and let $\cA = \{\bv_1, \ldots, \bv_n\} \subset V \setminus \{\mathbf{0}\}$ be a projective arc in $V = \F_q^k$ indexed by $\cI = \{1, \ldots, n\}$ of size $n = q + k - 1 - h$ with $n \ge h + k - 1$. Then there exists a global homogeneous polynomial $\Phi \in \F_q[X_1, \ldots, X_k]_h$ of degree $h$ in the coordinates of $V^*$ such that for every subset $\Lambda \subset \cI$ of size $k-2$, its restriction $\Phi|_{\Ann_{V^*}(U_\Lambda)}$ satisfies:
\begin{enumerate}
    \item $\Phi|_{\Ann_{V^*}(U_\Lambda)}$ is a nonzero homogeneous binary form of degree $h$ on $\Ann_{V^*}(U_\Lambda)$.
    \item All projective zeros of $\Phi|_{\Ann_{V^*}(U_\Lambda)}$ in $\PP(\Ann_{V^*}(U_\Lambda))$ are simple, and a point $[\phi] \in \PP(\Ann_{V^*}(U_\Lambda))$ satisfies $\Phi(\phi) = 0$ if and only if $\phi(\bv_i) \ne 0$ for all $i \in \Lambda^c$ (that is, $[\phi] \in M_\Lambda$).
    \item There exists a nonzero scalar $c_\Lambda \in \F_q^*$ such that for all $\phi \in \Ann_{V^*}(U_\Lambda)$,
    \[
        \Phi|_{\Ann_{V^*}(U_\Lambda)}(\phi) = c_\Lambda \prod_{j=1}^h M_{\Lambda,j}(\phi).
    \]
\end{enumerate}
\end{lemma}
\begin{proof}
In \cite[Theorem 13.1, pp.~155--157]{BL19} ,
 the multiplicity parameter $m \in \{1,2\}$ satisfies $m - 1 \equiv q \pmod 2$, so $m = 1$ when $q$ is even, and the number of tangents through a $(k-2)$-subset is $t = q + k - 1 - n = h$. Under the size hypothesis $n \ge h + k - 1$, Theorem~13.1 provides a homogeneous polynomial $\Phi$ of degree $h$ on $V^*$ after the cofactor substitution $Z_j = \det_j(X_1, \ldots, X_{k-1})$.

Specifically, fix a nonzero alternating $k$-linear form $\omega$ on $V$ (such as the determinant in the standard basis) and order $\Lambda = \{i_1, \ldots, i_{k-2}\}$. The cofactor map $V \to V^*$ defined by $\mathbf{x} \mapsto \omega(\bv_{i_1}, \ldots, \bv_{i_{k-2}}, \mathbf{x}, \cdot)$ has kernel $U_\Lambda$ and image $\Ann_{V^*}(U_\Lambda)$, inducing a linear isomorphism $\mathcal{C}_\Lambda : \Vbar{\Lambda} = V/U_\Lambda \xrightarrow{\sim} \Ann_{V^*}(U_\Lambda)$. It carries each missing quotient direction $[\mathbf{w}] \in \PP(\Vbar{\Lambda})$ to the dual point $[\phi] \in M_\Lambda$ of the corresponding tangent hyperplane $\PP(U_\Lambda + \Span_{\F_q}\{\mathbf{w}\})$. In Theorem~13.1, the polynomial identity $\Phi(\mathcal{C}_\Lambda(\mathbf{x})) = f_\Lambda(\mathbf{x})$ expresses the pullback of $\Phi|_{\Ann_{V^*}(U_\Lambda)}$ as the product $f_\Lambda$ of the $h$ distinct linear forms cutting out the tangent hyperplanes at $\Lambda$. Pushing forward along $\mathcal{C}_\Lambda$ yields the exact restriction assertions (1)--(3) on $\Ann_{V^*}(U_\Lambda)$.
\end{proof}

We recall the algebraic notions.
Let $\mathbb K$ be a field and let
$R=\mathbb K[X_1,\ldots,X_d]$, graded by total degree.
Write $R_r$ for the space of homogeneous polynomials of degree $r$,
with $R_r=\{0\}$ for $r<0$.
An ideal $\mathfrak a\subseteq R$ is \emph{homogeneous} if
\[
 \mathfrak a=\bigoplus_{r\ge0}\mathfrak a_r,
 \qquad
 \mathfrak a_r=\mathfrak a\cap R_r.
\]
For a homogeneous polynomial $P\in R$, the \emph{colon ideal}
$
 (\mathfrak a:P)=\{u\in R:Pu\in\mathfrak a\}
$
is also homogeneous.
A \emph{standard graded} $\mathbb K$-algebra is a finitely generated
commutative graded algebra $\cG=\bigoplus_{r\ge0}\cG_r$ with
$\cG_0=\mathbb K$ that is generated by $\cG_1$.
Equivalently, it is a quotient $R/\mathfrak a$ by a proper
homogeneous ideal, with
$
 \cG_r=R_r/\mathfrak a_r.$
Its \emph{embedding dimension} is $\dim_{\mathbb K}\cG_1$.
In this setting, $\cG$ is \emph{Artinian} if and only if it is
finite-dimensional over $\mathbb K$.
For such an Artinian algebra, its \emph{socle degree} is
$
 s=\max\{r:\cG_r\ne0\}.
$
Its homogeneous maximal ideal and its \emph{socle} are, respectively,
\[
 \mathfrak m_{\cG}=\bigoplus_{r\ge1}\cG_r,
 \qquad
 \operatorname{Soc}(\cG)
 =\{u\in\cG:\mathfrak m_{\cG}u=0\}.
\]
The socle is a graded ideal containing $\cG_s$.
The algebra $\cG$ is \emph{Gorenstein} if
$\dim_{\mathbb K}\operatorname{Soc}(\cG)=1$; in that case,
$\operatorname{Soc}(\cG)=\cG_s$.
A bilinear pairing
\[
 \mathcal U\times\mathcal W\longrightarrow\mathbb K
\]
between finite-dimensional $\mathbb K$-vector spaces is called
\emph{perfect} if every nonzero element of either space pairs
nontrivially with some element of the other.
Thus a perfect pairing is nondegenerate on both sides and satisfies
$\dim_{\mathbb K}\mathcal U=\dim_{\mathbb K}\mathcal W$.
Our critical-kernel estimate requires the following inequality
established by Zanello \cite[Theorem 4]{zanello}. 

\begin{lemma}\label{lem:unimodal}
Let $\cG$ be a standard graded Artinian Gorenstein algebra over a
field $\mathbb K$, with embedding dimension at most three and
socle degree $s$. Then
$
 \dim_{\mathbb K}\cG_r\ge\dim_{\mathbb K}\cG_{r-1}
$
for every integer $1\le r\le\lfloor s/2\rfloor.$
\end{lemma}
 
Recall that a linear $[n,k]$ MDS code is trivial if its dimension satisfies $k \in \{0, 1, n-1, n\}$. The basic parameter bounds and 
coordinate restriction properties for nontrivial MDS codes are established in \cite[Chapter 7]{HP03}.
\begin{lemma}\label{lem:mds-basic}
Let $\cC$ be a nontrivial linear $[n,k]$ MDS code over $\F_q$.
Then $\cC^\perp$ is MDS and
\[
 2\le k\le\min\{n-2,q-1\},\qquad n\le q+k-1.
\]
For any linear MDS code of dimension $d$, restriction to any
$r\le d$ coordinates has rank $r$; the subcode that is zero on
these coordinates therefore has dimension $d-r$.
\end{lemma}

\section{An extended Reed--Solomon code and its dual}
This section introduces a specialized linear space $E_{\cH}$, 
characterizes it as an extended Reed--Solomon code, and explicitly 
constructs a basis for its dual space.
Let $q=2^e$ with $e\ge 3$, $m=q/2$, and assume $1\le h\le m-1$.
Let $\cH \subset \F_q$ be a subset of size $|\cH| = h$, and let $\cH^c = \F_q \setminus \cH$ be its complement, of size $|\cH^c| = q - h = 2m - h$.
Define the corresponding monic vanishing polynomials:
\[
 P_{\cH}(X) = \prod_{\tau \in \cH} (X - \tau), \qquad P_{\cH^c}(X) = \prod_{s \in \cH^c} (X - s),
\]
so that $P_{\cH}(X) P_{\cH^c}(X) = X^q - X$.
For any data vector $\bz = \bigl((z_s)_{s \in \cH^c} \,;\, z_\infty\bigr)$ with $z_s, z_\infty \in \F_q$, define the associated interpolation polynomial
\[
 F(X) = \sum_{s \in \cH^c} z_s \frac{X^q - X}{X - s} = P_{\cH}(X) \sum_{s \in \cH^c} z_s \frac{P_{\cH^c}(X)}{X - s}.
\]
Then $\deg F \le q-1$, $F(\tau) = 0$ for all $\tau \in \cH$, and $F(s) = z_s$ for all $s \in \cH^c$, because $P_{\cH^c}'(s) = 1/P_{\cH}(s)$ in characteristic two.
We define $E_{\cH}$ as the $\F_q$-linear space of data vectors $\bz$ satisfying the first-order differential relation
$P_{\cH}(s)\bigl(F'(s) + z_\infty\bigr) = Q(s)$ on $\cH^c$ for some polynomial $Q\in\F_q[X]$ of degree strictly less than $h$; explicitly,
\begin{equation*}
 E_{\cH} = \left\{ \bz=\bigl((z_s)_{s \in \cH^c} \,;\, z_\infty\bigr) : \text{there is} ~ Q \in \F_q[X],\
 \deg Q < h,\ \frac{Q(s)}{P_{\cH}(s)} = F'(s) + z_\infty\ \text{for all} ~s \in \cH^c\right\}.
\end{equation*}
To characterize $E_{\cH}$ in closed form and identify its coding-theoretic structure, we recall the definition of an extended Reed--Solomon code.
Let $\Omega \subseteq \F_q$ be a subset of evaluation points, and let $1 \le r \le |\Omega| + 1$ be an integer.
The extended Reed--Solomon code of dimension $r$
evaluated over $\Omega \cup \{\infty\}$, denoted by $\cC_{\text{RS}}(\Omega, r)$, is the $\F_q$-linear space defined as:
\[
 \cC_{\text{RS}}(\Omega, r) =
 \left\{ \bigl( (f(x))_{x \in \Omega} \,;\, f_{r-1} \bigr): f(X) =
 \sum_{j=0}^{r-1} f_j X^j \in \F_q[X] \right\}.
\]

The following lemma shows that $E_{\cH}$ is parameterized by polynomials containing only even powers of $X$ together with a linear term
and coincides with an extended Reed--Solomon code over the squared evaluation set $(\cH^c)^2$.

\begin{lemma}\label{lem:space}
The linear space $E_{\cH}$ is explicitly given by
\[
 E_{\cH} = \left\{ \bigl( (E(s) + a s)_{s \in \cH^c} \,;\, a \bigr) \;\middle\vert{}\; E \in \Span_{\F_q}\{1, X^2, \ldots, X^{q-2}\},\ a \in \F_q \right\}.
\]
Consequently, $\dim_{\F_q} E_{\cH} = m+1$. Moreover, $E_{\cH}=\cC_{\text{RS}}((\cH^c)^2, m+1)$, where $(\cH^c)^2 = \{s^2 \mid s \in \cH^c\}$.
\end{lemma}
\begin{proof}
Let $E_{\text{param}}$ denote the set on the right-hand side of the lemma.
We first prove $E_{\cH} \subseteq E_{\text{param}}$.
Let $\bz \in E_{\cH}$. By the definition of $E_{\cH}$, there is a polynomial $F(X)$ with $\deg(F) \le q-1$ 
satisfying $F(\tau) = 0$ for all $\tau \in \cH$,
$F(s) = z_s$ for all $s \in \cH^c$, and there is a polynomial $Q(X)$ with $\deg(Q) < h$ 
satisfying $Q(s) = P_{\cH}(s)(F'(s) + z_\infty)$ for all $s \in \cH^c$.
Let $a = z_\infty$ and $A(X) = F'(X) + a$.
Since $P_{\cH}(s)A(s) - Q(s) = 0$ for all $s \in \cH^c$,
the vanishing polynomial $P_{\cH^c}(X) = \prod_{s \in \cH^c} (X-s)$ must divide $P_{\cH}(X)A(X) - Q(X)$.
Thus there is a polynomial $G(X) \in \F_q[X]$ such that
$
P_{\cH}(X)A(X) - Q(X) = P_{\cH^c}(X)G(X).
$
Since $\deg A = \deg F' \le q-2$ and $\deg Q < h$, comparing degrees yields $\deg G \le (h+q-2) - (q-h) = 2h-2$.
Multiply the identity by $P_{\cH}(X)$ and substitute $P_{\cH}(X)P_{\cH^c}(X) = X^q - X$:
\[
    P_{\cH}(X)^2 A(X) - P_{\cH}(X)Q(X) = (X^q - X)G(X).
\]
Apply formal differentiation $\frac{d}{dX}$ over $\F_q$.
Since $F'(X)$ contains only even powers,
$
    A'(X) = F''(X) + a' = 0,
$
and
$
    (P_{\cH}^2 A)'  = 0.
$
Differentiating $P_{\cH}^2 A - P_{\cH}Q = (X^q - X)G$, we have
\[
    (P_{\cH}Q)' = G + (X^q - X)G'.
\]
Analyze the degrees of both sides. For the left-hand side,
$\deg(P_{\cH}Q)' \le 2h - 2$.
Assume $G' \neq 0$. For the right-hand side:
$
    \deg\bigl((X^q - X)G'\bigr) = q + \deg(G') \ge q.
$
Given the parameter bound $h \le \frac{q}{2} - 1$, we have
\[
    \deg(P_{\cH}Q)' \le 2\left(\frac{q}{2} - 1\right) - 2 = q - 4.
\]
The assumption $G' \neq 0$ creates a degree contradiction.
Thus $G' = 0$, which yields $G = (P_{\cH}Q)'$.
Substitute $G$ back to isolate $A(X)$,
giving
\[
A = \frac{P_{\cH}Q + (X^q - X)(P_{\cH}Q)'}{P_{\cH}^2} = \frac{Q}{P_{\cH}} + (X^q - X) \left(\frac{Q}{P_{\cH}}\right)'.
\]
Since $\deg(Q) < \deg(P_{\cH})$ and $P_{\cH}(X) = \prod_{\tau \in \cH} (X-\tau)$, apply partial fractions:
\[
    \frac{Q(X)}{P_{\cH}(X)} = \sum_{\tau \in \cH} \frac{c_\tau}{X-\tau}, \quad \text{where } c_\tau = \frac{Q(\tau)}{P_{\cH}'(\tau)}.
\]
Substitute this into $A(X)$:
\[
    A(X) = \sum_{\tau \in \cH} c_\tau \left( \frac{1}{X-\tau} + (X^q - X) \frac{1}{(X-\tau)^2} \right)
= \sum_{\tau \in \cH} c_\tau (X-\tau)^{q-2}.
\]
From $A(X) = F'(X) + a$, we have
\[
    F'(X) = a + \sum_{\tau \in \cH} c_\tau (X-\tau)^{q-2}.
\]
For polynomials of degree $\le q-1$, the kernel of $\frac{d}{dX}$ consists precisely of even-powered polynomials:
\[
    \ker\left(\frac{d}{dX}\right) = \Span_{\F_q}\{1, X^2, X^4, \dots, X^{q-2}\}.
\]
Integrating $F'(X)$ yields the form of $F(X)$:
\[
    F(X) = U(X) + aX + \sum_{\tau \in \cH} c_\tau (X-\tau)^{q-1}, \quad \text{where } U(X) \in \Span_{\F_q}\{1, X^2, \dots, X^{q-2}\}.
\]
Define $c_0 =\sum_{\tau \in \cH} c_\tau$ and let $E(X)= U(X) + c_0$. Since $c_0$ is a constant, $E(X) \in \Span_{\F_q}\{1, X^2, \dots, X^{q-2}\}$.
Evaluate $F(s)$ for $s \in \cH^c$:
\[
    z_s = F(s) = U(s) + as + \sum_{\tau \in \cH} c_\tau = U(s) + c_0 + as = E(s) + as.
\]
Because $z_s = E(s) + as$ and $z_\infty = a$, we conclude $\bz \in E_{\text{param}}$. This proves $E_{\cH} \subseteq E_{\text{param}}$.

We next prove $E_{\text{param}} \subseteq E_{\cH}$.
Let $\bz \in E_{\text{param}}$. Then $z_\infty = a$ and $z_s = E(s) + as$ for some $E \in \Span_{\F_q}\{1, X^2, \ldots, X^{q-2}\}$ and $a \in \F_q$.
Define constants $c_\tau = E(\tau) + a\tau$ for all $\tau \in \cH$, and let $c_0 = \sum_{\tau \in \cH} c_\tau$.
Define the purely even-powered polynomial $U(X) = E(X) + c_0$.
Construct the polynomial $F(X)$:
\[
    F(X) = U(X) + aX + \sum_{\tau \in \cH} c_\tau (X-\tau)^{q-1}.
\]
For any $\tau_0 \in \cH$, we have
\[
    F(\tau_0) = U(\tau_0) + a\tau_0 + \sum_{\tau \in \cH, \tau \neq \tau_0} c_\tau = U(\tau_0) + a\tau_0 + c_0 - c_{\tau_0},
\]
giving  $F(\tau_0) = 0$.
For any $s \in \cH^c$,  we have
\[
    F(s) = U(s) + as + \sum_{\tau \in \cH} c_\tau = U(s) + c_0 + as = E(s) + as = z_s.
\]
Since $\deg F \le q-1$ and $\vert{}\cH \cup \cH^c\vert{} = q$, this polynomial $F(X)$ coincides with 
the unique interpolation polynomial associated with $\bz$.
Construct the polynomial $Q(X) = P_{\cH}(X) \sum_{\tau \in \cH} \frac{c_\tau}{X-\tau}$, which ensures $\deg(Q) < \deg(P_{\cH}) = h$.
Calculate the formal derivative of $F(X)$:
\[
    F'(X) = a + \sum_{\tau \in \cH} c_\tau (X-\tau)^{q-2}.
\]
Evaluate $P_{\cH}(s)(F'(s) + z_\infty)$ at $s \in \cH^c$:
\[
    P_{\cH}(s)(F'(s) + a) = P_{\cH}(s) \left( a + \sum_{\tau \in \cH} c_\tau (s-\tau)^{q-2} + a \right) = P_{\cH}(s) \sum_{\tau \in \cH} c_\tau (s-\tau)^{q-2},
\]
yielding
\[
    P_{\cH}(s) \sum_{\tau \in \cH} \frac{c_\tau}{s-\tau} = Q(s).
\]
Because there exist valid polynomials $F(X)$ and $Q(X)$ satisfying all defining constraints, $\bz \in E_{\cH}$. 
This proves $E_{\text{param}} \subseteq E_{\cH}$, establishing $E_{\cH} = E_{\text{param}}$.

Having established $E_{\cH} = E_{\text{param}}$, we determine its dimension.
The space $E_{\text{param}}$ is parameterized by $a \in \F_q$ and $E(X) = \sum_{j=0}^{m-1} e_j X^{2j}$, 
which provides exactly $m+1$ independent coefficients.
Assume $(E, a)$ maps to the zero vector, meaning $z_\infty = a = 0$ and $z_s = E(s) + as = 0$ for all $s \in \cH^c$.
Because $a = 0$, we have $E(s) = 0$ for all $s \in \cH^c$.
Since $E(X)$ contains only even powers and $\deg(E) \le 2m-2$, 
there exists $K(X) \in \F_q[X]$ with $\deg K\leq m-1$ such that $E(X) = K(X)^2$.
For all $s \in \cH^c$, we have $E(s) = K(s)^2 = 0$, yielding $K(s) = 0$.
Since $\deg(K) \le m-1 < m+1 \le |\cH^c|$, this forces $E(X) = 0$.
The parameterization mapping is strictly injective, so no dimensions are lost:
\[
    \dim_{\F_q} E_{\cH} = \dim_{\F_q} E_{\text{param}} = m + 1.
\]

Finally, we prove that $E_{\text{param}} \subseteq \cC_{\text{RS}}((\cH^c)^2, m+1)$.
Let $\bz \in E_{\text{param}}$. The coordinates are $z_\infty = a$ and $z_s = E(s) + as$, 
where $E(X) = \sum_{j=0}^{m-1} e_j X^{2j}$.
Letting $x = s^2$, 
the value $E(s)$ becomes:
\[
    E(s) = \sum_{j=0}^{m-1} e_j (s^2)^j = \sum_{j=0}^{m-1} e_j x^j.
\]
The coordinates evaluate to
\[
    z_s = E(s) + as = \sum_{j=0}^{m-1} e_j x^j + a x^m.
\]
Define $f(X) = \sum_{j=0}^{m-1} e_j X^j + a X^m$. Clearly $\deg f \le m$.
We have $z_s = f(s^2)$ for all $s \in \cH^c$, and $z_\infty = a = f_m$.
Thus $\bz \in \cC_{\text{RS}}((\cH^c)^2, m+1)$, proving $E_{\text{param}} \subseteq \cC_{\text{RS}}((\cH^c)^2, m+1)$.
The strict equality holds by dimension.
\end{proof}

We now construct an explicit basis for the dual space
\[
    E_{\cH}^\perp = \left\{ \by = \bigl((y_s)_{s \in \cH^c} \,;\, y_\infty\bigr) \in \F_q^{\cH^c \cup \{\infty\}} \;\middle\vert\; \sum_{s \in \cH^c} y_s z_s + y_\infty z_\infty = 0 \text{ for all } \bz \in E_{\cH} \right\}.
\]
The ambient vector space $\F_q^{\cH^c \cup \{\infty\}}$ has dimension $|\cH^c| + 1 = 2m - h + 1$.
By Lemma~\ref{lem:space}, $\dim_{\F_q} E_{\cH} = m + 1$, so the dimension of the dual space is
\[
    t = \dim_{\F_q}(E_{\cH}^\perp) = (2m - h + 1) - (m + 1) = m - h \ge 1.
\]
For each index $0 \le r \le t-1$, we define the vector $\bb^{(r)} = \bigl((b^{(r)}_s)_{s \in \cH^c} \,;\, b^{(r)}_\infty\bigr) \in \F_q^{\cH^c \cup \{\infty\}}$ by
\begin{equation}\label{eq:dualrows}
    b^{(r)}_s = P_{\cH}(s)^2 s^{2r} \quad (\text{for all } s \in \cH^c), \qquad
    b^{(r)}_\infty = \begin{cases}
        0, & 0 \le r \le t-2, \\
        1, & r = t-1.
    \end{cases}
\end{equation}

\begin{proposition}\label{prop:dual}
The $t$ vectors $\bb^{(0)}, \bb^{(1)}, \dots, \bb^{(t-1)}$ defined in \eqref{eq:dualrows}  form a basis for $E_{\cH}^\perp$.
\end{proposition}
\begin{proof}
The cardinality of $(\cH^c)^2 = \{s^2 \mid s \in \cH^c\}$ is $|(\cH^c)^2| = |\cH^c| = q - h = 2m - h$.
By Lemma~\ref{lem:space}, we have $E_{\cH} = \cC_{\text{RS}}((\cH^c)^2, m+1)$.
Consider the monic vanishing polynomials for $\cH^c$ and $(\cH^c)^2$:
\[
    P_{\cH^c}(X) = \prod_{c \in \cH^c} (X - c), \qquad P_{(\cH^c)^2}(X) =  \prod_{c \in \cH^c} (X - c^2).
\]
For $s\in\cH^c$, we have
\[
    P_{\cH^c}'(s) = \prod_{c \in \cH^c \setminus \{s\}} (s - c)
~\text{and}~
 P_{(\cH^c)^2}'(s^2) = \prod_{c \in \cH^c \setminus \{s\}} (s^2 - c^2) =  P_{\cH^c}'(s)^2.
\]
From the factorization $X^q - X = P_{\cH}(X)P_{\cH^c}(X)$, formal differentiation yields
$
    1 = P_{\cH}'(X)P_{\cH^c}(X) + P_{\cH}(X)P_{\cH^c}'(X).
$
Evaluating at $s \in \cH^c$ gives $1 = P_{\cH}(s)P_{\cH^c}'(s)$. 
Squaring both sides yields $P_{\cH^c}'(s)^2 = P_{\cH}(s)^{-2}$, and therefore
$
    P_{(\cH^c)^2}'(s^2) = P_{\cH}(s)^{-2}.
$
Setting $x = s^2 \in (\cH^c)^2$, the finite coordinates of $\bb^{(r)}$ can be rewritten as
\[
    b^{(r)}_s = P_{\cH}(s)^2 (s^2)^r = \frac{x^r}{P_{(\cH^c)^2}'(x)}.
\]

Consider the Lagrange interpolation of the monomial $X^d$ over the set $(\cH^c)^2$ for $0 \le d \le q - h - 1$:
\[
    X^d = \sum_{x \in (\cH^c)^2} x^d \frac{P_{(\cH^c)^2}(X)}{(X-x)P_{(\cH^c)^2}'(x)}.
\]
Since $P_{(\cH^c)^2}(X)$ is a monic polynomial of degree $q - h$, each quotient $\frac{P_{(\cH^c)^2}(X)}{X-x}$ is a monic polynomial of degree $q - h - 1$. Comparing the coefficients of $X^{q-h-1}$ on both sides yields the standard Euler--Lagrange identity:
\[
    \sum_{x \in (\cH^c)^2} \frac{x^d}{P_{(\cH^c)^2}'(x)} = \begin{cases} 0, & 0 \le d \le q - h - 2, \\ 1, & d = q - h - 1. \end{cases}
\]

Let $\bz \in E_{\cH}$. By Lemma~\ref{lem:space}, $\bz$ is parameterized by a polynomial 
$f(X) = \sum_{j=0}^m f_j X^j$ with $\deg f\le m$ such that $z_s = f(s^2)$ for all $s \in \cH^c$ and $z_\infty = f_m$.
The bilinear pairing between $\bb^{(r)}$ and $\bz$ is
\[
    \langle \bb^{(r)}, \bz \rangle = \sum_{s \in \cH^c} b^{(r)}_s z_s + b^{(r)}_\infty z_\infty = \sum_{x \in (\cH^c)^2} \frac{x^r f(x)}{P_{(\cH^c)^2}'(x)} + b^{(r)}_\infty f_m.
\]
Expanding $f(x)$ and exchanging the order of summation gives
\[
    \sum_{x \in (\cH^c)^2} \frac{x^r f(x)}{P_{(\cH^c)^2}'(x)} = \sum_{j=0}^m f_j \left( \sum_{x \in (\cH^c)^2} \frac{x^{r+j}}{P_{(\cH^c)^2}'(x)} \right).
\]
Let $d_j = r+j$, which has maximum value $d_{\max} = r + m$ as $0 \le j \le m$.
We evaluate the pairing according to the index $0 \le r \le t-1$:

\textbf{Case 1:} $0 \le r \le t-2$ (when $t \ge 2$).
Here the maximum exponent is bounded by
$
    d_{\max} \le (t-2) + m  = q-h-2.
$
Thus $d_j \le q-h-2$ for all $0 \le j \le m$, so all inner sums vanish by the Euler--Lagrange identity. Since $b^{(r)}_\infty = 0$, we obtain
$
    \langle \bb^{(r)}, \bz \rangle = 0.
$

\textbf{Case 2:} $r = t-1$.
Here the maximum exponent is
$
    d_{\max} = (t-1) + m = q-h-1.
$
The exponent $d_j = q-h-1$ is attained if and only if $j = m$. By the Euler--Lagrange identity, 
the inner sum vanishes for all $j < m$ and equals $1$ for $j = m$, so the finite coordinate sum equals $f_m$. 
Since $b^{(t-1)}_\infty = 1$, we obtain
$
    \langle \bb^{(t-1)}, \bz \rangle = 0.
$
Consequently, $\langle \bb^{(r)}, \bz \rangle = 0$ for all $0 \le r \le t-1$, proving that $\bb^{(r)} \in E_{\cH}^\perp$.
To prove linear independence, suppose that
\[
    \sum_{r=0}^{t-1} \lambda_r \bb^{(r)} = \mathbf{0}
\]
for some scalars $\lambda_0, \ldots, \lambda_{t-1} \in \F_q$.
Evaluating the infinity coordinate gives
$\lambda_{t-1} = 0$.
If $t = 1$, this already completes the proof. If $t \ge 2$, evaluating the finite coordinates at $s \in \cH^c$ 
and omitting the $r = t-1$ term yields
$
    \sum_{r=0}^{t-2} \lambda_r P_{\cH}(s)^2 (s^2)^r = 0 
$
for all $s \in \cH^c$.
Since $s \notin \cH$, we have $P_{\cH}(s) \neq 0$. Dividing by $P_{\cH}(s)^2$ 
and setting $x = s^2$ gives
$
    \sum_{r=0}^{t-2} \lambda_r x^r = 0 
$
for all  $x \in (\cH^c)^2$.
Thus the polynomial $G(X) = \sum_{r=0}^{t-2} \lambda_r X^r$ 
has degree at most $t-2 = m-h-2$, while vanishing on all $|(\cH^c)^2| = 2m-h$ 
elements of $(\cH^c)^2$.
Since $2m - h > m - h - 2 \ge \deg(G)$, the polynomial $G(X)$ 
must be identically zero, so $\lambda_r = 0$ for all $0 \le r \le t-2$.
Therefore, the $t$ vectors $\bb^{(0)}, \ldots, \bb^{(t-1)}$ are 
linearly independent. Since $\dim_{\F_q}(E_{\cH}^\perp) = t$, they form a basis for $E_{\cH}^\perp$.
\end{proof}

\section{Algebraic properties for the dual code}
To link the space $E_{\cH}$ with the generator matrix of the arc $\cA$, we extract a first-order differential identity on $V^*$ directly
from the local factorization of the homogeneous polynomial $\Phi$ provided by Lemma~\ref{lem:input}.
Fix a basis of $V^*$ to identify elements of $V^*$ with coordinate vectors in $\F_q^k$. 
For any polynomial $G \in \F_q[X_1, \ldots, X_k]$ on $V^*$, 
its formal directional differential at a base point 
$\phi = (\phi_1, \ldots, \phi_k) \in V^*$ 
in the direction $\eta = (\eta_1, \ldots, \eta_k) \in V^*$ is defined by
\[
 dG_\phi(\eta) = \sum_{j=1}^k \eta_j \frac{\partial G}{\partial X_j}(\phi) \in \F_q,
\]
where $\partial G / \partial X_j$ denotes the ordinary formal 
partial derivative over $\F_q$. Equivalently, $dG_\phi(\eta)$ is the 
coefficient of the linear term $T$ in the polynomial expansion 
$G(\phi + T\eta) \in \F_q[T]$; consequently, 
its value is independent of the choice of basis on $V^*$, 
and the map $\eta \mapsto dG_\phi(\eta)$ is $\F_q$-linear on $V^*$.

\begin{lemma}\label{lem:gradient}
Let $k \ge 4$ be an integer,  $h = k - 3$ and $n = q + 2$ with $n \ge h + k - 1$. 
Let $\cA = \{\bv_1, \ldots, \bv_n\} \subset V \setminus \{\mathbf{0}\}$ 
be a projective arc in $V = \F_q^k$ indexed by $\cI = \{1, \ldots, n\}$, and let 
$\Phi \in \F_q[X_1, \ldots, X_k]_h$ be the homogeneous 
envelope polynomial of degree $h$ on $V^*$ provided by Lemma~\ref{lem:input}. 
For any $(k-1)$-subset $\Lambda' \subset \cI$ and every nonzero linear functional 
$\phi \in \Ann_{V^*}(U_{\Lambda'})$, one has $\Phi(\phi) \ne 0$ 
and $\phi(\bv_i) \ne 0$ for all $i \notin  \Lambda'$, 
and the formal differential of $\Phi$ at $\phi$ satisfies
\begin{equation}\label{eq:gradient}
 \frac{d\Phi_\phi(\eta)}{\Phi(\phi)}
 = \sum_{i \notin \Lambda'} \frac{\eta(\bv_i)}{\phi(\bv_i)}
 \qquad\text{for every }\eta \in V^*.
\end{equation}
\end{lemma}
\begin{proof}
We first claim that $\phi(\bv_i) \ne 0$ for every $i \notin \Lambda'$. 
Suppose for contradiction that there exists some $i \notin \Lambda'$ such that $\phi(\bv_i) = 0$. 
Then $\phi$ vanishes on the $k$ columns indexed by $\Lambda' \cup \{i\}$. By the definition of an arc, 
we have $\phi = 0$, which is a contradiction. 

Fix an arbitrary element $a \in \Lambda'$ and define the subset $\Lambda = \Lambda' \setminus \{a\}$.  
We have $\phi \in \Ann_{V^*}(U_\Lambda)$.  By Lemma~\ref{lem:input}, 
the roots of the restricted polynomial $\Phi|_{\Ann_{V^*}(U_\Lambda)}$ are exactly the missing points $M_\Lambda$, which are 
characterized by the strict condition: $[\xi]$ is a root if and only if $\xi(\bv_i) \ne 0$ for every $i \notin \Lambda$. However, 
evaluating $\phi$ at the column $\bv_a$ (where $a \notin \Lambda$), we know $\phi(\bv_a) = 0$ because $a \in \Lambda'$. 
This means $[\phi]$ fails the necessary condition to be a missing point,  yielding $\Phi(\phi) \ne 0$.

For each $i \notin \Lambda$, 
define the linear form $L_i(\xi) = \xi(\bv_i)$ on $\Ann_{V^*}(U_\Lambda)$. The roots of these linear forms account for 
$|\Lambda^c| = n - k + 2 = q + 1 - h$ distinct points on $\PP(\Ann_{V^*}(U_\Lambda))$. Because $\Phi|_{\Ann_{V^*}(U_\Lambda)}$ 
evaluates to nonzero on the roots of $L_i$, these two sets of roots are strictly disjoint. Therefore, the product polynomial
\[
 \left.\left(\Phi\prod_{i\notin \Lambda}L_i\right)\right|_{\Ann_{V^*}(U_\Lambda)}
\]
has degree $h + (q + 1 - h) = q + 1$, and its roots exactly cover all $q+1$ rational points of $\PP(\Ann_{V^*}(U_\Lambda))$, each with algebraic multiplicity one.

We now establish a local coordinate system $(X,Y)$ on $\Ann_{V^*}(U_\Lambda)$. Choose the first basis vector to be $\alpha = \phi$, so that $\phi$ has coordinates $(1,0)$. 
We can choose the second basis vector $\beta$ such that $L_a(\beta) \ne 0$, making the linear form $L_a$ strictly proportional to the coordinate $Y$, 
meaning $L_a(X,Y) = cY$ for some $c \in \F_q^*$.
In finite fields, the unique (up to a scalar multiple) degree-$(q+1)$ binary form with all $q+1$ projective points as simple roots is the Dickson form $X^qY - XY^q$. 
By the fundamental theorem of algebra, our product polynomial must be a nonzero scalar multiple of this form:
\[
 \left.\left(\Phi\prod_{i\notin \Lambda}L_i\right)\right|_{\Ann_{V^*}(U_\Lambda)} = c_0(X^qY - XY^q), \qquad c_0 \in \F_q^*.
\]
Since $\Lambda = \Lambda' \setminus \{a\}$, 
we split the product over $i \notin \Lambda$ into the factor $L_a$ 
and the product over $i \notin \Lambda'$. 
Substituting $L_a = cY$ and factoring $Y$ on the right side yields:
\[
 cY \cdot \left.\left(\Phi\prod_{i\notin \Lambda'}L_i\right)\right|_{\Ann_{V^*}(U_\Lambda)} = c_0 Y (X^q - XY^{q-1}).
\]
Defining $c_\Lambda = c_0/c \ne 0$, we obtain the exact algebraic identity:
\[
 \left.\left(\Phi\prod_{i\notin \Lambda'}L_i\right)\right|_{\Ann_{V^*}(U_\Lambda)} = c_\Lambda(X^q - XY^{q-1})
\]
where $c_\Lambda \ne 0$.
Let $G_\Lambda(X,Y) = c_\Lambda(X^q - XY^{q-1})$. 
We compute its formal partial derivatives with respect to $X$ and $Y$:
\[
 \frac{\partial G_\Lambda}{\partial X} = c_\Lambda(qX^{q-1} - Y^{q-1}), \qquad \frac{\partial G_\Lambda}{\partial Y} = -c_\Lambda(q-1)XY^{q-2}.
\]
We evaluate these at the base point $\phi$, which corresponds to coordinates $(X,Y) = (1,0)$,
leading to
\[
 \frac{\partial G_\Lambda}{\partial X}(1,0) =  0, 
 \qquad \frac{\partial G_\Lambda}{\partial Y}(1,0) = 0.
\]
Therefore, the first formal differential of $G_\Lambda$ at $\phi$ in any direction 
$\eta = (\eta_X, \eta_Y) \in \Ann_{V^*}(U_\Lambda)$ 
is a linear combination of these partial derivatives:
\[
 d(G_\Lambda)_\phi(\eta) = \eta_X \frac{\partial G_\Lambda}{\partial X}(1,0) + 
 \eta_Y \frac{\partial G_\Lambda}{\partial Y}(1,0) = 0.
\]
Let $L_{\Lambda'} = \prod_{i\notin \Lambda'}L_i$. The established identity on $\Ann_{V^*}(U_\Lambda)$ is $\Phi \cdot L_{\Lambda'} = G_\Lambda$. Taking the formal differential of both sides at $\phi$ in direction $\eta$ yields 
\[
 d\Phi_\phi(\eta) \cdot L_{\Lambda'}(\phi) = \Phi(\phi) \cdot d(L_{\Lambda'})_\phi(\eta).
\]
Dividing both sides by the nonzero product $\Phi(\phi) L_{\Lambda'}(\phi)$ yields:
\[
 \frac{d\Phi_\phi(\eta)}{\Phi(\phi)} = \frac{d(L_{\Lambda'})_\phi(\eta)}{L_{\Lambda'}(\phi)}.
\]
By the product rule for multiple factors, the logarithmic 
derivative of a product expands to the sum of the logarithmic derivatives of its components:
\[
 \frac{d(L_{\Lambda'})_\phi(\eta)}{L_{\Lambda'}(\phi)} = \sum_{i\notin \Lambda'}\frac{d(L_i)_\phi(\eta)}{L_i(\phi)}.
\]
For the linear form $L_i(\xi) = \xi(\bv_i)$, its formal differential strictly equals itself evaluated at the direction vector, $d(L_i)_\phi(\eta) = \eta(\bv_i)$. Substituting this alongside $L_i(\phi) = \phi(\bv_i)$ into the sum proves the identity for all $\eta \in \Ann_{V^*}(U_\Lambda)$:
\[
 \frac{d\Phi_\phi(\eta)}{\Phi(\phi)} = \sum_{i\notin \Lambda'}\frac{\eta(\bv_i)}{\phi(\bv_i)}, 
 \qquad \text{for all}~ \eta \in \Ann_{V^*}(U_\Lambda).
\]

Observe that both sides of the identity in \eqref{eq:gradient} 
are linear with respect to the direction variable $\eta$. 
Since $a \in \Lambda'$ was chosen arbitrarily, the identity holds for any $\eta \in \Ann_{V^*}(U_{\Lambda'\setminus\{a\}})$ and every $a \in \Lambda'$.
To demonstrate that this holds globally across $V^*$, pick an index $b \in \cI \setminus \Lambda'$ and complete the 
linearly independent columns indexed by $\Lambda'$
to a full basis of $V$ by adjoining the column $\bv_b$. Let $\{\mathbf{e}_x^* : x \in \Lambda'\} \cup \{\mathbf{e}_b^*\}$ be the corresponding dual basis of $V^*$.
We have 
$\mathbf{e}_a^* \in \Ann_{V^*}(U_{\Lambda'\setminus\{a\}})$ 
and  $\mathbf{e}_b^* \in \Ann_{V^*}(U_{\Lambda'\setminus\{a\}})$.
Consequently, taking the union of these subspaces over all $a \in \Lambda'$, 
the combined set $\bigcup_{a \in \Lambda'} \Ann_{V^*}(U_{\Lambda'\setminus\{a\}})$ 
contains $\mathbf{e}_b^*$ and every $\mathbf{e}_a^*$ ($a \in \Lambda'$).
Because this union contains the entire dual basis of $V^*$, its linear span covers the whole dual space $V^*$. Since the identity equates two linear mappings that agree on a set of subspaces spanning $V^*$, the mappings must be unconditionally identical on the entire space. This proves the identity for arbitrary $\eta \in V^*$, including directions transverse to any one of these two-dimensional subspaces.
\end{proof}

The connection between the abstract space $E_{\cH}$ and a linear MDS code $\cC$ is established by mapping 
the restricted rows of the generator matrix into $E_{\cH}$. This relies on a precise normalization of the quotient coordinates 
over the locator set $\cH^c$ and the differential identity from Lemma~\ref{lem:gradient}.
Let $\Lambda \subset \cI$ with $| \cI|=q+2$ and $|\Lambda| = k-2$. 
Choose one specific column index $\infty \notin \Lambda$. 
Since $\Ann_{V^*}(U_\Lambda)$ cannot vanish identically on $\bv_\infty$ by the arc property, we may choose a basis $\{\alpha, \beta\}$ of $\Ann_{V^*}(U_\Lambda)$ such that $\alpha(\bv_\infty) = 0$ and $\beta(\bv_\infty) = y \ne 0$.
We apply a nonzero coordinate rescaling 
to the column vector
$
 \btv_\infty = y^{-1} \bv_\infty,
$
yielding
$(\alpha(\btv_\infty), \beta(\btv_\infty)) = (0, 1).$

For any remaining index $i \in \cI \setminus (\Lambda \cup \{\infty\})$, 
we must have $\alpha(\bv_i) = x_i \ne 0$. 
Setting $y_i = \beta(\bv_i)$ and rescaling the column via $\btv_i = x_i^{-1} \bv_i$, we obtain
\[
 (\alpha(\btv_i), \beta(\btv_i)) = (1, s_i), \qquad \text{where } s_i = y_i x_i^{-1} \in \F_q.
\]
Moreover, if $s_i = s_{i'}$ for some $i, i' \in \cI \setminus (\Lambda \cup \{\infty\})$, then 
$\alpha(\btv_i - \btv_{i'}) = 0$ and $\beta(\btv_i - \btv_{i'}) = 0$, meaning $\btv_i - \btv_{i'}$ 
is annihilated by 
$\Ann_{V^*}(U_\Lambda)$ and hence lies in $U_\Lambda$. 
By the linear independence of any $k$ columns of the arc, 
this requires $i = i'$. Thus, as $i$ ranges over the $(q+2) - (k-2) - 1 = q - h$ indices in 
$\cI \setminus (\Lambda \cup \{\infty\})$ ($h = k - 3$), 
the resulting scalars $s_i \in \F_q$ are pairwise distinct. 
We denote this set of $q - h$ affine scalars by
\[
 \cH^c = \{ s_i \in \F_q : i \in \cI \setminus (\Lambda \cup \{\infty\}) \},
\]
and define its complement in $\F_q$ as $\cH = \F_q \setminus \cH^c$, 
which has size $|\cH| = h$. Re-indexing these normalized affine columns by 
$s \in \cH^c$ (writing $\btv_s$ in place of $\btv_i$ when $s_i = s$), the evaluations take the normalized form
\[
 (\alpha(\btv_s), \beta(\btv_s)) = (1, s) \quad \text{for all } s \in \cH^c.
\]

Every functional $\xi \in \Ann_{V^*}(U_\Lambda)$ is uniquely parameterized by 
coordinates $(X,Y) \in \F_q^2$ as $\xi = X\alpha + Y\beta$. 
By Lemma~\ref{lem:input}, a projective point 
$[\xi] \in \PP(\Ann_{V^*}(U_\Lambda))$ is a zero of $\Phi|_{\Ann_{V^*}(U_\Lambda)}$ 
 if and only if $\xi(\btv_\infty) \ne 0$ and $\xi(\btv_s) \ne 0$ for all $s \in \cH^c$. 
 Evaluating $\xi = X\alpha + Y\beta$ at these normalized columns  gives
\[
 \xi(\btv_\infty) = Y, \qquad \xi(\btv_s) = X + sY  \quad (s \in \cH^c).
\]
Consequently, the missing point set $M_\Lambda \subset \PP(\Ann_{V^*}(U_\Lambda))$ does not contain the point at infinity $[\alpha]$, and all $h$ projective zeros of $\Phi|_{\Ann_{V^*}(U_\Lambda)}$ lie in the affine chart where $Y \ne 0$.
Dehomogenizing by setting $Y = 1$, every point in this affine chart is uniquely represented by 
$\xi_X = X\alpha + \beta$ for $X \in \F_q$. Define the univariate polynomial
\[
 H_0(X) = \Phi(X\alpha + \beta) \in \F_q[X].
\]
Writing the homogeneous binary form $\Phi(X\alpha + Y\beta)$ as $\sum_{j=0}^h \tilde{c}_j X^j Y^{h-j}$, the leading coefficient of $X^h$ in $H_0(X)$ is $\tilde{c}_h = \Phi(\alpha) \ne 0$, so $\deg H_0(X) = h$. Furthermore, for any $X \in \F_q$, we have $(X\alpha + \beta)(\btv_\infty) = 1$ and $(X\alpha + \beta)(\btv_s) = X - s$. By Lemma~\ref{lem:input} again, $H_0(X) = 0$ if and only if $X - s \ne 0$ for all $s \in \cH^c$, which is equivalent to $X \in \F_q \setminus \cH^c = \cH$. Hence, the root set of $H_0(X)$ in $\F_q$ is precisely $\cH$, and each of the $h$ elements of $\cH$ is a simple root. Multiplying the global form $\Phi$ by the nonzero scalar $\tilde{c}_h^{-1} \in \F_q^*$, we normalize $\Phi$ so that $\Phi(X\alpha + \beta)$ coincides with the monic vanishing polynomial $P_{\cH}(X) = \prod_{\tau \in \cH} (X - \tau)$:
\[
 P_{\cH}(X) = \Phi(X\alpha + \beta) = X^h + c_{h-1}X^{h-1} + \dots + c_0.
\]
Let $\eta \in V^*$ be an arbitrary  linear functional.
The formal differential of $\Phi$ at the base point $X\alpha + \beta$ in the direction $\eta$ is:
\[
 Q_\eta(X) = d\Phi_{X\alpha+\beta}(\eta) = \sum_{j=1}^k \eta_j \frac{\partial \Phi}{\partial X_j}(X\alpha + \beta).
\]
It follows that
$
 \deg Q_\eta(X) \le h-1.
$
Fix $s \in \cH^c$ and define the base point functional $\phi_s = s\alpha + \beta$.
Define the subset $\Lambda' = \Lambda \cup \{s\}$. 
Evaluate $\phi_s$ on the columns of $\Lambda'$:
1. For $i \in \Lambda$, $\phi_s(\bv_i) = 0$ because $\alpha, \beta \in \Ann_{V^*}(U_\Lambda)$.
2. For the specific column $\btv_s$, $\phi_s(\btv_s) = 
s\alpha(\btv_s) + \beta(\btv_s) = 0$.
Thus, $\phi_s \in \Ann_{V^*}(U_{\Lambda'})$ annihilates the $(k-1)$-set $\Lambda'$.
Since each ratio $\eta(\bv_i)/\phi_s(\bv_i)$ is invariant under nonzero scalar rescaling of $\bv_i$, substituting $\phi_s$ and the normalized columns into Lemma~\ref{lem:gradient} gives
\[
 \frac{d\Phi_{\phi_s}(\eta)}{\Phi(\phi_s)} = \sum_{i \notin \Lambda'} \frac{\eta(\btv_i)}{\phi_s(\btv_i)}.
\]
$\Phi(\phi_s) = \Phi(s\alpha + \beta) = P_{\cH}(s)$ and $d\Phi_{\phi_s}(\eta) = Q_\eta(s)$.
Thus, $ \frac{d\Phi_{\phi_s}(\eta)}{\Phi(\phi_s)} = \frac{Q_\eta(s)}{P_{\cH}(s)}$.
The set of indices $i \notin \Lambda'$ is exactly $(\cH^c \setminus \{s\}) \cup \{\infty\}$.
Define $z_u = \eta(\btv_u)$ for $u \in \cH^c$, and $z_\infty = \eta(\btv_\infty)$.
For affine columns $u \in \cH^c \setminus \{s\}$,  we have
$
 \phi_s(\btv_u) = s+u;
$
for the infinity column, 
$
 \phi_s(\btv_\infty)  = 1.
$
Summing these discrete parts yields:
\[
 \frac{Q_\eta(s)}{P_{\cH}(s)} = \sum_{u \in \cH^c \setminus \{s\}} \frac{z_u}{s+u} + z_\infty.
\]
Recall the global interpolation polynomial for the  space $E_{\cH}$:
\[
 F(X) = \sum_{u \in \cH^c} z_u \frac{X^q-X}{X-u}.
\]
We formally differentiate $F(X)$ and evaluate at $X=s$. We split the sum into $u=s$ and $u \ne s$.

\textbf{Case 1: The $u=s$ summand.}
Divide by $(X-s)$:
\[
\frac{X^q-X}{X-s} = \frac{(X-s)^q + (X-s)}{X-s} = (X-s)^{q-1} + 1.
\]
Differentiate this quotient with respect to $X$:
\[
 \frac{d}{dX} \left[ z_s ((X-s)^{q-1} + 1) \right] = z_s (q-1)(X-s)^{q-2}.
\]
Evaluate at $X=s$ giving $0$.

\textbf{Case 2: The $u \ne s$ summands.}
We have
\[
 \frac{d}{dX} \left( z_u \frac{X^q-X}{X-u} \right) = z_u \frac{1 \cdot (X-u) - (X^q-X) \cdot 1}{(X-u)^2}.
\]
Evaluating  at $X=s$,  the derivative is exactly $\frac{z_u}{s+u}$.
Summing the derivatives gives
\[
 F'(s) = 0 + \sum_{u \in \cH^c \setminus \{s\}} \frac{z_u}{s+u}.
\]
We therefore arrive at 
$
 \frac{Q_\eta(s)}{P_{\cH}(s)} = F'(s) + z_\infty,
$
or equivalently $Q_\eta(s) = P_{\cH}(s) \left( F'(s) + z_\infty \right).$
This holds for every $s \in \cH^c$. Since $\deg Q_\eta < h$, the data vector $\bz = (z_s : s \in \cH^c ; z_\infty)$ perfectly satisfies the algebraic definition of the  space $E_{\cH}$.
Because $\eta \in V^*$ is an arbitrary functional (not restricted to $\Ann_{V^*}(U_\Lambda)$), the vector $\bz$ represents the evaluation of an arbitrary linear combination of the rows of the normalized generator matrix. Thus, every row of this restricted generator matrix inherently belongs to $E_{\cH}$.

Consequently, the vectors in $E_{\cH}^\perp$ describe the dual constraints of the code under these local normalized coordinates. 
Reversing this normalization allows us to express the dual codewords globally in terms of the coordinate-free properties of the arc.

Let $m=q/2\geq4$, $4 \le k \le m + 1$, $h = k-3$, and $t = m - h$. 
Let $\cC$ be a linear $[q+2, k]$ MDS code over $\F_q$ with fixed nonzero generator 
columns $\bv_i \in V = \F_q^k$ for $i \in \cI$ where $|\cI| = q+2$. 
For any index subset $\Lambda \subset \cI$ of size $k-2$,
recall the subspace $U_\Lambda = \Span_{\F_q}\{\bv_i : i \in \Lambda\}$
and the $2$-dimensional quotient space $\Vbar{\Lambda} = V / U_\Lambda$. 
The projected columns $\bp_i = \bv_i + U_\Lambda \in \Vbar{\Lambda}$ ($i \notin \Lambda$) define $q+1-h$ distinct points in $\PP(\Vbar{\Lambda})$, leaving exactly $h$ missing projective directions. 
Let $H_\Lambda$ be a nonzero homogeneous binary form of degree $h$ on $\Vbar{\Lambda}$ whose simple zeros in $\PP(\Vbar{\Lambda})$ are precisely these missing directions. 
The following proposition formalizes this coordinate-free algebraic translation.

\begin{proposition}\label{prop:actual}
As $f$ ranges over the $t$-dimensional space of all binary homogeneous forms of degree $t-1$ on $\Vbar{\Lambda}$, the vectors $\bb^{(f)} \in \F_q^{\cI}$ defined by
\[
 b^{(f)}_i = 0 \quad (i \in \Lambda), \qquad
 b^{(f)}_i = \bigl( H_\Lambda(\bp_i) f(\bp_i) \bigr)^2 \quad (i \notin \Lambda),
\]
form a $t$-dimensional linear subcode $\cD_\Lambda \subseteq \cC^\perp$.
\end{proposition}
\begin{proof}
We divide the proof into two steps: first establishing that $\cD_\Lambda$ is a $t$-dimensional 
$\F_q$-linear subspace of $\F_q^{\cI}$, and then proving that 
$\cD_\Lambda$ is orthogonal to $\cC$.

Let $\mathcal{F}_{t-1}$ denote the $t$-dimensional $\F_q$-vector space of all 
binary homogeneous forms of degree $t-1$ on the $2$-dimensional quotient space 
$\Vbar{\Lambda}$, and consider the evaluation map $\Psi : \mathcal{F}_{t-1} \to \F_q^{\cI}$ 
given by $f \mapsto \bb^{(f)}$. 
Since $\Psi(f_1 + f_2) = \Psi(f_1) + \Psi(f_2)$ and $\Psi(\lambda f) = \lambda^2 \Psi(f)$ for all $\lambda \in \F_q$, the map $\Psi$ is Frobenius-semilinear; because $\lambda \mapsto \lambda^2$ is an automorphism of $\F_q$, the image $\cD_\Lambda = \Psi(\mathcal{F}_{t-1})$ is an $\F_q$-linear subspace of $\F_q^{\cI}$.
To determine $\dim_{\F_q} \cD_\Lambda$, suppose that $f \in \mathcal{F}_{t-1}$ 
satisfies $\bb^{(f)} = \mathbf{0}$. Then for every $i \notin \Lambda$,
$
H_\Lambda(\bp_i) f(\bp_i) = 0.
$
By definition, the zeros of $H_\Lambda$ in $\PP(\Vbar{\Lambda})$ lie strictly in the complementary set $\PP(\Vbar{\Lambda}) \setminus \{[\bp_i] : i \notin \Lambda\}$. Hence $H_\Lambda(\bp_i) \ne 0$ for all $i \notin \Lambda$, which forces
$
 f(\bp_i) = 0
$
for all $i \notin \Lambda$.
Consequently, the binary homogeneous form $f$ of degree $t-1$ vanishes at the $q+1-h$ distinct projective points $\{[\bp_i] : i \notin \Lambda\} \subseteq \PP(\Vbar{\Lambda})$. Since $t = m - h$ and $m = q/2 \ge 4$, we have
\[
 \deg f = t - 1 = m - h - 1 < 2m + 1 - h = q + 1 - h,
\]
forcing $f = 0$. Therefore, the Frobenius-semilinear map $\Psi$ is injective, yielding
$
 \dim_{\F_q} \cD_\Lambda = \dim_{\F_q} \mathcal{F}_{t-1} = t.
$

We now turn to the next step of the proof. 
We identify $\Vbar{\Lambda}$ with $\F_q^2$ using the basis 
$\{\alpha, \beta\}$ of $\Ann_{V^*}(U_\Lambda)$ fixed prior to the proposition, 
and relate $H_\Lambda$ to the monic vanishing polynomial 
$P_{\cH}(X) = \prod_{\tau \in \cH}(X - \tau)$.
Since $\alpha$ and $\beta$ vanish on $U_\Lambda$, the map
\[
 \bp = \bv + U_\Lambda \longmapsto (T_1, T_2) = \bigl(\alpha(\bv), 
 \beta(\bv)\bigr) \qquad (\bv \in V)
\]
is a well-defined linear isomorphism $\Vbar{\Lambda} \xrightarrow{\sim} \F_q^2$ with homogeneous coordinates $(T_1, T_2)$.
By the column normalization established prior to the proposition, the projected columns $\bp_i = \bv_i + U_\Lambda$ for $i \notin \Lambda$ take the form:
\begin{equation}\label{eq:proj_coords}
 \bp_\infty = (0, y) = y(0, 1), \qquad
 \bp_i = (x_i, x_i s_i) = x_i(1, s_i) \quad \text{for } i \in \cI \setminus (\Lambda \cup \{\infty\}),
\end{equation}
where $y = \beta(\bv_\infty) \in \F_q^*$, $x_i = \alpha(\bv_i) \in \F_q^*$, and $s_i = x_i^{-1}\beta(\bv_i)$ ranges bijectively over $\cH^c = \F_q \setminus \cH$.
In $\PP(\Vbar{\Lambda})$, the $q+1-h$ occupied points $\{[\bp_i] : i \notin \Lambda\}$ are $[0:1]$ and $\{[1:s] : s \in \cH^c\}$, so the $h$ missing projective directions are precisely $\{[1:\tau] : \tau \in \cH\}$.
Since each point $[1:\tau]$ is defined by $T_2 - \tau T_1 = 0$, the degree-$h$ homogeneous form $H_\Lambda(T_1, T_2)$ vanishing at these $h$ points must take the form
\[
 H_\Lambda(T_1, T_2) = c \prod_{\tau \in \cH} (T_2 - \tau T_1) \quad \text{for some } c \in \F_q^*.
\]
Evaluating $H_\Lambda$ on the normalized direction vectors $(0,1)$ and $(1,s)$ for $s \in \cH^c$ yields
\begin{equation}\label{eq:H_eval}
 H_\Lambda(0, 1) = c \prod_{\tau \in \cH} 1 = c, \qquad
 H_\Lambda(1, s) = c \prod_{\tau \in \cH} (s - \tau) = c P_{\cH}(s).
\end{equation}

\smallskip\noindent
Next, we choose a scaled monomial basis of $\mathcal{F}_{t-1}$ so that the squared product $(H_\Lambda f_r)^2$ is homogeneous of degree $q-2$, which turns scalar multiplication by $\lambda \in \F_q^*$ into multiplication by $\lambda^{-1}$.
Define the basis $\{f_0, f_1, \dots, f_{t-1}\}$ of $\mathcal{F}_{t-1}$ by
$
 f_r(T_1, T_2) = c^{-1} T_1^{t-1-r} T_2^r
$
for 
$0 \le r \le t-1$, 
and set $G_r(T_1, T_2) = \bigl(H_\Lambda(T_1, T_2) f_r(T_1, T_2)\bigr)^2$. Since $\deg H_\Lambda = h$, $\deg f_r = t-1$, and $h + t = m = q/2$, the form $G_r$ is homogeneous of degree
\[
 \deg G_r = 2(h + t - 1) = 2(m - 1) = q - 2.
\]
Because $\lambda^{q-2} = \lambda^{-1}$ for every $\lambda \in \F_q^*$, homogeneity of $G_r$ implies that for all $\lambda \in \F_q^*$ and $(T_1, T_2) \in \F_q^2$,
\begin{equation}\label{eq:homog_inv}
 G_r(\lambda T_1, \lambda T_2) = \lambda^{q-2} G_r(T_1, T_2) = \lambda^{-1} \bigl(H_\Lambda(T_1, T_2) f_r(T_1, T_2)\bigr)^2.
\end{equation}

\smallskip\noindent
We now evaluate the components $b^{(f_r)}_i = G_r(\bp_i)$ for $i \notin \Lambda$ by combining \eqref{eq:proj_coords}, \eqref{eq:H_eval}, and \eqref{eq:homog_inv}, and match them with the basis vectors $\bb^{(r)} = \bigl((b^{(r)}_s)_{s \in \cH^c}; b^{(r)}_\infty\bigr) \in E_{\cH}^\perp$ defined in \eqref{eq:dualrows}:
\begin{itemize}
    \item For $i = \infty$, substituting $\bp_\infty = y(0, 1)$ into \eqref{eq:homog_inv} and using \eqref{eq:H_eval} gives
    \begin{equation}\label{eq:b_inf}
     b^{(f_r)}_\infty = G_r\bigl(y(0, 1)\bigr)
     = y^{-1} \bigl(H_\Lambda(0, 1) f_r(0, 1)\bigr)^2
     = y^{-1} \bigl(c \cdot c^{-1} 0^{t-1-r} 1^r\bigr)^2
     = y^{-1} b^{(r)}_\infty,
    \end{equation}
    where $b^{(r)}_\infty = 0$ for $0 \le r \le t-2$ and $b^{(t-1)}_\infty = 1$.
    \item For $i \in \cI \setminus (\Lambda \cup \{\infty\})$ with $s_i = s \in \cH^c$, substituting $\bp_i = x_i(1, s)$ into \eqref{eq:homog_inv} and using \eqref{eq:H_eval} gives
    \begin{equation}\label{eq:b_aff}
     b^{(f_r)}_i = G_r\bigl(x_i(1, s)\bigr)
     = x_i^{-1} \bigl(H_\Lambda(1, s) f_r(1, s)\bigr)^2
     = x_i^{-1} \bigl(c P_{\cH}(s) \cdot c^{-1} s^r\bigr)^2
     = x_i^{-1} P_{\cH}(s)^2 s^{2r}
     = x_i^{-1} b^{(r)}_s.
    \end{equation}
\end{itemize}

\smallskip\noindent

Finally, we show that $\bb^{(f_r)}$ is orthogonal to every codeword $\mathbf{c} = (\eta(\bv_i))_{i \in \cI} \in \cC$ ($\eta \in V^*$).
By the derivation preceding the proposition, the evaluation vector $\bz = \bigl((z_s)_{s \in \cH^c}; z_\infty\bigr)$ on the normalized columns $\btv_s = x_i^{-1}\bv_i$ ($s_i = s$) and $\btv_\infty = y^{-1}\bv_\infty$, defined by
\[
 z_s = \eta(\btv_s) = x_i^{-1}\eta(\bv_i) \quad (s \in \cH^c), \qquad
 z_\infty = \eta(\btv_\infty) = y^{-1}\eta(\bv_\infty),
\]
belongs to $E_{\cH}$. Equivalently, the unnormalized codeword components satisfy
\begin{equation}\label{eq:code_unnorm}
 \eta(\bv_i) = x_i z_s \quad (s_i = s \in \cH^c), \qquad
 \eta(\bv_\infty) = y z_\infty.
\end{equation}
Since $b^{(f_r)}_i = 0$ for all $i \in \Lambda$, substituting \eqref{eq:b_inf}, \eqref{eq:b_aff}, and \eqref{eq:code_unnorm} into the inner product over $\cI$ cancels the scaling factors $x_i$ and $y$:
\[
\begin{aligned}
 \langle \bb^{(f_r)}, \mathbf{c} \rangle
 &= \sum_{i \in \cI \setminus (\Lambda \cup \{\infty\})} b^{(f_r)}_i \eta(\bv_i) + b^{(f_r)}_\infty \eta(\bv_\infty) \\
 &= \sum_{s \in \cH^c} \bigl(x_i^{-1} b^{(r)}_s\bigr)(x_i z_s) + \bigl(y^{-1} b^{(r)}_\infty\bigr)(y z_\infty) \\
 &= \sum_{s \in \cH^c} b^{(r)}_s z_s + b^{(r)}_\infty z_\infty \\
 &= \langle \bb^{(r)}, \bz \rangle = 0,
\end{aligned}
\]
where the last equality holds by Proposition~\ref{prop:dual} since $\bb^{(r)} \in E_{\cH}^\perp$ and $\bz \in E_{\cH}$.
Because $\mathbf{c} \in \cC$ is arbitrary, $\bb^{(f_r)} \in \cC^\perp$ for all $0 \le r \le t-1$. Since $\{f_0, \dots, f_{t-1}\}$ spans $\mathcal{F}_{t-1}$ and the map $f \mapsto \bb^{(f)}$ is Frobenius-semilinear, we conclude that $\cD_\Lambda \subseteq \cC^\perp$.
\end{proof}

A subscript
$d$ on a polynomial ring denotes its space of homogeneous forms of
degree $d$, with this space understood to be zero for $d<0$.
Fix $\Lambda_0\subset\cI$ with $|\cI|=q+2$ and $|\Lambda_0|=h=k-3$, 
and let $\Lambda_0^c=\cI\setminus\Lambda_0$.
Define
\[
 \operatorname{Short}_{\Lambda_0}(\cC^\perp)
 =\{\bb\in\cC^\perp:b_i=0\text{ for every }i\in\Lambda_0\}.
\]
Proposition~\ref{prop:actual} gives
\begin{equation}\label{eq:common-shortening}
 \sum_{a\in\Lambda_0^c}\cD_{\Lambda_0\cup\{a\}}
 \subseteq\operatorname{Short}_{\Lambda_0}(\cC^\perp),
\end{equation}
since every summand vanishes on $\Lambda_0$. 
Consider the coordinate projection map 
$\cC^\perp \longrightarrow \F_q^{\Lambda_0}$ defined by 
$\bb \longmapsto (b_i)_{i\in\Lambda_0}$, which has rank $h$. 
Since the kernel of this linear mapping corresponds precisely to the shortened code over 
$\Lambda_0$, the rank-nullity theorem yields
$\dim\operatorname{Short}_{\Lambda_0}(\cC^\perp) = q-2h-1$.
Let
\[
 \Vbar{\Lambda_0}=V/U_{\Lambda_0},\qquad
 \overline{\bv}_i=\bv_i+U_{\Lambda_0}\quad(i\in\Lambda_0^c),\qquad
 \overline{\cA}=\{\overline{\bv}_i:i\in\Lambda_0^c\}.
\]
Then $\dim\Vbar{\Lambda_0}=3$ and 
$\overline{\cA}$ is a planar arc of size $q+2-h$.
For each fixed $a\in\Lambda_0^c$, retain the notation as before
\[
 \bp_i=\bv_i+U_{\Lambda_0\cup\{a\}}
 \in\Vbar{\Lambda_0\cup\{a\}}\quad(i\in\Lambda_0^c\setminus\{a\}),
\]
and let $\mathcal F_d$ be the space of degree-$d$ binary forms on
this quotient. Deleting the zero coordinates in $\Lambda_0$,
Proposition~\ref{prop:actual} reads
\begin{equation}\label{eq:planar-canonical-code}
 \begin{split}
 \cD_{\Lambda_0\cup\{a\}}
 &=\{\bb^{(f)}:f\in\mathcal F_{t-1}\}
 \subseteq\F_q^{\Lambda_0^c},\\
 b_a^{(f)}&=0,\qquad
 b_i^{(f)}=
 \bigl(H_{\Lambda_0\cup\{a\}}(\bp_i)f(\bp_i)\bigr)^2
 \quad(i\ne a).
 \end{split}
\end{equation}
Deleting these identically zero coordinates preserves dimensions
and sums; we use the same notation for the resulting subcodes.
Choose coordinates $(X_1,X_2,X_3)$ on $(\Vbar{\Lambda_0})^*$ and define
$
 L_a(\phi)=\phi(\overline{\bv}_a)
 $
 for
 $
a\in\Lambda_0^c.
$
By the planar arc property of $\overline{\cA}$, the linear forms $L_a$ are pairwise nonproportional, meaning that $L_a \notin \F_q^* L_b$ (or $\PP(\ker L_a) \ne \PP(\ker L_b)$) 
for all distinct $a, b \in \Lambda_0^c$, and the isomorphism 
$(\Vbar{\Lambda_0})^* \cong \Ann_{V^*}(U_{\Lambda_0})$ identifies
\[
 \ker L_a = \bigl\{\phi \in (\Vbar{\Lambda_0})^* : \phi(\overline{\bv}_a) 
 = 0\bigr\} \cong \Ann_{V^*}(U_{\Lambda_0\cup\{a\}}).
\]
Under this identification, let $M_{\Lambda_0\cup\{a\}} = 
\bigl\{ [\phi] \in \PP(\ker L_a) : \phi(\overline{\bv}_i) \ne 0 \text{ for all } i \in \Lambda_0^c \setminus \{a\} \bigr\} = \bigl\{ [\phi_{a,1}], \ldots, [\phi_{a,h}] \bigr\}$ be the $h$ 
missing dual points from Section~2, and write
\[
 M_{a,j} = M_{\Lambda_0\cup\{a\},j} : \ker L_a \longrightarrow \F_q, 
 \qquad \ker M_{a,j} = \Span_{\F_q}\{\phi_{a,j}\} \qquad (1 \le j \le h)
\]
for the corresponding nonzero binary linear forms cutting out these 
missing points on $\ker L_a$.
We identify polynomial restriction to $\ker L_a$ with reduction
modulo $L_a$.
For $\bz=(z_i)_{i\in\Lambda_0^c}$, define
\begin{equation}\label{eq:W-definition}
 W_{\bz}=\sum_{i\in\Lambda_0^c}z_iL_i^{q-2}
 \in\F_q[X_1,X_2,X_3]_{q-2}.
\end{equation}

With the global polynomial $W_{\bz}$ constructed, we now establish a crucial algebraic criterion. The following lemma demonstrates that the condition of orthogonality to the local canonical subcode $\cD_{\Lambda_0\cup\{a\}}$ translates cleanly into an explicit subspace membership test for the restricted polynomial $W_{\bz} \bmod L_a$.

\begin{lemma}\label{lem:coordinates}
Let $a\in\Lambda_0^c$ and $\bz\in\F_q^{\Lambda_0^c}$. The vector $\bz$ belongs 
to the dual space $\cD_{\Lambda_0\cup\{a\}}^\perp$ if and only if
$
 W_{\bz}\bmod L_a\in \Span_{\F_q}\{M_{a,1}^{q-2},\ldots,M_{a,h}^{q-2}\}.
$
Furthermore, the $h$ polynomials $M_{a,1}^{q-2}, \ldots, M_{a,h}^{q-2}$ generating this span are linearly independent.
\end{lemma}
\begin{proof}
Fix an index $a \in \Lambda_0^c$ and equip 
$\Vbar{\Lambda_0\cup\{a\}}$ with a coordinate system $(T_1, T_2)$, 
such that any vector $\bp \in \Vbar{\Lambda_0\cup\{a\}}$ is represented by 
its components $(p_1, p_2)$. By identifying $\ker L_a$ with  
$(\Vbar{\Lambda_0\cup\{a\}})^*$, we assign it the corresponding dual coordinates $(U_1, U_2)$. Consequently, 
$
 \phi(\bp) = p_1U_1 + p_2U_2.
$
Furthermore, restricting the global linear form $L_i$, which is originally defined on 
$(\Vbar{\Lambda_0})^*$,
to the subspace $\ker L_a$ algebraically corresponds to evaluating the functional $\phi$ precisely at the projected column vector $\bp_i \in \Vbar{\Lambda_0\cup\{a\}}$, yielding 
$
 L_i|_{\ker L_a}(\phi) = \phi(\bp_i)
$
for all $i \in \Lambda_0^c \setminus \{a\}$. Here and throughout the proof, all sums indexed by $i \ne a$ are understood to run over $i \in \Lambda_0^c \setminus \{a\}$.

Under $(\ker L_a)^* \cong \Vbar{\Lambda_0\cup\{a\}}$, 
each missing-point linear form $M_{a,j} \in (\ker L_a)^*$ ($1 \le j \le h$) corresponds to a unique nonzero vector $\bw_{a,j} = (w_{a,j,1}, w_{a,j,2}) \in \Vbar{\Lambda_0\cup\{a\}}$ satisfying
\[
 M_{a,j}(\phi) = \phi(\bw_{a,j}) = w_{a,j,1}U_1 + w_{a,j,2}U_2 \qquad (\phi \in \ker L_a).
\]
These vectors $\bw_{a,1}, \ldots, \bw_{a,h}$ represent the $h$ missing quotient directions in $\PP(\Vbar{\Lambda_0\cup\{a\}})$. Indeed, if $[\phi_{a,j}] \in \PP(\ker L_a)$ is the projective zero of $M_{a,j}$, then $\phi_{a,j}(\bw_{a,j}) = M_{a,j}(\phi_{a,j}) = 0$, whereas $\phi_{a,j}(\bp_i) \ne 0$ for every $i \ne a$ by the definition of a missing dual point. This ensures that $[\bw_{a,j}] \ne [\bp_i]$ for every occupied direction $[\bp_i]$. Moreover, because the $h$ missing dual points $[\phi_{a,1}], \ldots, [\phi_{a,h}]$ are pairwise distinct, the linear forms $M_{a,1}, \ldots, M_{a,h}$ are pairwise nonproportional, which implies that the corresponding missing directions $[\bw_{a,1}], \ldots, [\bw_{a,h}]$ are also pairwise distinct. For any scalar $z \in \F_q$, we write $z^{1/2} = z^m$ to denote its unique square root under the inverse Frobenius automorphism of $\F_q$.

Define the monomials $g_r(T_1,T_2) = T_1^{m-1-r}T_2^r \in \mathcal{F}_{m-1}$ for $0 \le r \le m-1$. Homogenizing the identity $(1+T)^{m-1} = \sum_{r=0}^{m-1} T^r$,
for any vector $\bp = (p_1, p_2) \in \Vbar{\Lambda_0\cup\{a\}}$ gives
\[
 (p_1U_1 + p_2U_2)^{m-1} = \sum_{r=0}^{m-1} g_r(\bp) U_1^{m-1-r}U_2^r.
\]
Squaring both sides in characteristic two and noting that $2(m-1) = q-2$, we obtain the polynomial expansion
\begin{equation}\label{eq:binary-power-expansion}
 (p_1U_1+p_2U_2)^{q-2}
 =\sum_{r=0}^{m-1}g_r(\bp)^2U_1^{q-2-2r}U_2^{2r}.
\end{equation}
Applying \eqref{eq:binary-power-expansion} to both $L_i|_{\ker L_a}(\phi) = \phi(\bp_i)$ and $M_{a,j}(\phi) = \phi(\bw_{a,j})$, we can expand the restricted polynomial $W_{\bz} \bmod L_a = \sum_{i \ne a} z_i (L_i|_{\ker L_a})^{q-2}$ and any linear combination $\sum_{j=1}^h \mu_j M_{a,j}^{q-2}$ (with $\mu_1, \ldots, \mu_h \in \F_q$) in the monomial basis of $\F_q[U_1, U_2]_{q-2}$:
\[
 W_{\bz} \bmod L_a = \sum_{r=0}^{m-1} \left(\sum_{i \ne a} z_i g_r(\bp_i)^2\right) U_1^{q-2-2r}U_2^{2r}
\]
and 
\[
 \sum_{j=1}^h \mu_j M_{a,j}^{q-2} = \sum_{r=0}^{m-1} \left(\sum_{j=1}^h \mu_j g_r(\bw_{a,j})^2\right) U_1^{q-2-2r}U_2^{2r}.
 \]
Since the monomials $\{U_1^{q-2-2r}U_2^{2r}\}_{r=0}^{m-1}$ are linearly independent in $\F_q[U_1, U_2]_{q-2}$, the polynomial identity
\[
 W_{\bz}\bmod L_a = \sum_{j=1}^h\mu_jM_{a,j}^{q-2}
\]
holds if and only if the coefficients of each monomial coincide, that is, if and only if
\[
 \sum_{i\ne a}z_i g_r(\bp_i)^2 = \sum_{j=1}^h\mu_j g_r(\bw_{a,j})^2 \qquad \text{for all } 0 \le r \le m-1.
\]
Applying the inverse Frobenius automorphism $z \mapsto z^{1/2}$ to both sides of each equation, this system is equivalent to
\[
 \sum_{i\ne a}z_i^{1/2} g_r(\bp_i) = \sum_{j=1}^h\mu_j^{1/2} g_r(\bw_{a,j}) \qquad \text{for all } 0 \le r \le m-1.
\]
Because the monomials $g_0, \ldots, g_{m-1}$ form a basis for the space $\mathcal{F}_{m-1}$ of degree-$(m-1)$ binary homogeneous forms on $\Vbar{\Lambda_0\cup\{a\}}$, extending by linearity shows that the above relation holds if and only if
\begin{equation}\label{eq:evaluation-functional-identity}
 \sum_{i\ne a}z_i^{1/2}g(\bp_i)
 =\sum_{j=1}^h\mu_j^{1/2}g(\bw_{a,j})
 \qquad\text{for all } g\in\mathcal F_{m-1}.
\end{equation}
Consider the linear multi-point evaluation map at the $h$ missing directions:
\[
 \mathcal E_a:\mathcal F_{m-1}\longrightarrow\F_q^h,
 \qquad g\longmapsto\bigl(g(\bw_{a,1}),\ldots,g(\bw_{a,h})\bigr).
\]
We first verify that $\mathcal{E}_a$ is surjective. For each $1 \le l \le h$, let $L_{\bw_{a,l}} \in \mathcal{F}_1$ be a nonzero linear form whose projective zero is $[\bw_{a,l}]$. For a fixed index $1 \le j \le h$, choose a linear form $\widetilde{L}_j \in \mathcal{F}_1$ that does not vanish at $\bw_{a,j}$, and construct the homogeneous polynomial
\[
 P_j = \widetilde{L}_j^{m-h} \prod_{\substack{1 \le l \le h \\ l \ne j}} L_{\bw_{a,l}} \in \mathcal{F}_{m-1},
\]
which has degree $(m-h) + (h-1) = m-1$. Since the missing directions $[\bw_{a,1}], \ldots, [\bw_{a,h}]$ are pairwise distinct, $P_j$ vanishes at $\bw_{a,l}$ for all $l \ne j$ while evaluating to a nonzero scalar $c_j = P_j(\bw_{a,j}) \in \F_q^*$ at $\bw_{a,j}$. Rescaling $P_j$ by $c_j^{-1}$ yields an element of $\mathcal{F}_{m-1}$ whose image under $\mathcal{E}_a$ is the $j$-th standard basis vector of $\F_q^h$. As $j$ is arbitrary, $\mathcal{E}_a$ is surjective.

Next, we characterize $\ker \mathcal{E}_a$. A binary homogeneous form $g \in \mathcal{F}_{m-1}$ lies in $\ker \mathcal{E}_a$ if and only if $g(\bw_{a,j}) = 0$ for every $1 \le j \le h$, which holds if and only if each linear factor $L_{\bw_{a,j}}$ divides $g$. Since the $h$ missing directions are pairwise distinct, the linear forms $L_{\bw_{a,1}}, \ldots, L_{\bw_{a,h}}$ are pairwise coprime in the unique factorization domain $\F_q[T_1, T_2]$, and their product is a nonzero scalar multiple of the missing-direction form $H_{\Lambda_0\cup\{a\}} \in \mathcal{F}_h$. Consequently, $g \in \ker \mathcal{E}_a$ if and only if $H_{\Lambda_0\cup\{a\}}$ divides $g$ in $\F_q[T_1, T_2]$. Recalling that $t = m - h$, we obtain the exact kernel decomposition
\begin{equation}\label{eq:missing-kernel}
 \ker\mathcal E_a
 =H_{\Lambda_0\cup\{a\}}\mathcal F_{m-1-h}
 =H_{\Lambda_0\cup\{a\}}\mathcal F_{t-1}.
\end{equation}

Define the linear evaluation functional $\ell_{\bz,a} : \mathcal{F}_{m-1} \to \F_q$ associated with the occupied directions by
\[
 \ell_{\bz,a}(g)=\sum_{i\ne a}z_i^{1/2}g(\bp_i) \qquad (g \in \mathcal{F}_{m-1}).
\]
Observe that the right-hand side of \eqref{eq:evaluation-functional-identity} represents the composition of $\mathcal{E}_a$ with a linear functional on $\F_q^h$. Because
$\mathcal{E}_a : \mathcal{F}_{m-1} \to \F_q^h$ is surjective, 
there exist scalars $\mu_1, \ldots, \mu_h \in \F_q$ satisfying \eqref{eq:evaluation-functional-identity} if and only if $\ell_{\bz,a}$ factors through $\mathcal{E}_a$, which holds if and only if $\ell_{\bz,a}$ vanishes identically on $\ker \mathcal{E}_a$.
By the kernel identity \eqref{eq:missing-kernel}, $\ell_{\bz,a}$ vanishes on $\ker \mathcal{E}_a$ 
if and only if $\ell_{\bz,a}(H_{\Lambda_0\cup\{a\}}f) = 0$ for all $f \in \mathcal{F}_{t-1}$,
which, upon squaring both sides, becomes
\[
 \sum_{i\ne a}z_i
 \bigl(H_{\Lambda_0\cup\{a\}}(\bp_i)
 f(\bp_i)\bigr)^2=0
 \qquad\text{for all } f\in\mathcal F_{t-1}.
\]
We have shown that 
$\bz\in \cD_{\Lambda_0\cup\{a\}}^\perp$ if and only if
$
 W_{\bz}\bmod L_a\in \Span_{\F_q}\{M_{a,1}^{q-2},\ldots,M_{a,h}^{q-2}\}.
$

Finally, suppose that $\sum_{j=1}^h \mu_j M_{a,j}^{q-2} = 0$ for some scalars $\mu_1, \ldots, \mu_h \in \F_q$. Setting $\bz = \mathbf{0}$ (so that $W_{\bz} = 0$) in the equivalence leading to \eqref{eq:evaluation-functional-identity} yields
\[
 \sum_{j=1}^h \mu_j^{1/2} g(\bw_{a,j}) = 0 \qquad \text{for all } g \in \mathcal{F}_{m-1}.
\]
Since $\mathcal{E}_a$ is surjective, the evaluation tuple $\bigl(g(\bw_{a,1}), \ldots, g(\bw_{a,h})\bigr)$ ranges over the entire space $\F_q^h$ as $g$ varies over $\mathcal{F}_{m-1}$. This forces $\mu_j^{1/2} = 0$, and hence $\mu_j = 0$, for every $1 \le j \le h$, proving that the $h$ polynomials $M_{a,1}^{q-2}, \ldots, M_{a,h}^{q-2}$ are linearly independent.
\end{proof}

We now pass from a single line $\ker L_a$ to the arrangement of all lines 
for $a \in \Lambda_0^c$.
By the planar arc property of $\overline{\cA}$, any three distinct projected 
columns $\overline{\bv}_a, \overline{\bv}_b, \overline{\bv}_c$ are linearly 
independent and thus form a basis for  $\Vbar{\Lambda_0}$. 
Consequently, 
in the dual projective plane $\PP((\Vbar{\Lambda_0})^*)$, 
we have
$
 \ker L_a \cap \ker L_b \cap \ker L_c = \{\mathbf{0}\} 
 $
for any distinct  $a,b,c \in \Lambda_0^c$.
For a fixed index $a \in \Lambda_0^c$ and any distinct index $b \in \Lambda_0^c \setminus \{a\}$, 
we have
$
 \dim(\ker L_a \cap \ker L_b) = 1.
$
Thus, each other line $\ker L_b$ intersects the subspace $\ker L_a$ 
at a unique projective point $[\ker L_a \cap \ker L_b] \in \PP(\ker L_a)$. 
Furthermore,  the mapping
\[
 b \longmapsto [\ker L_a \cap \ker L_b]
\]
is strictly injective. As $b$ ranges over all other indices in 
$\Lambda_0^c \setminus \{a\}$, the set of occupied directions
$
 \bigl\{ [\ker L_a \cap \ker L_b] \bigm| b \in \Lambda_0^c \setminus \{a\} \bigr\}
$
yields exactly $|\Lambda_0^c| - 1 = q+1-h$ pairwise distinct intersection points on the projective line $\PP(\ker L_a)$. 
Meanwhile, the $h$ pairwise nonproportional forms 
$M_{a,1}, \ldots, M_{a,h}$ define $h$ distinct projective zeros $[\phi_{a,j}] \in \PP(\ker L_a)$. 
By the definition of missing points, these functionals do not vanish on any occupied direction, 
ensuring their zeros are completely disjoint from the intersections. 
Cardinality matching then forces the $h$ remaining rational 
points to be exactly the zeros of $M_{a,1}, \ldots, M_{a,h}$.

Since $q+2-h\ge h+2$, Lemma~\ref{lem:input} applies to
$\overline{\cA}$. 
Let $\Phi_0 \in \F_q[X_1,X_2,X_3]_h$ be the homogeneous 
polynomial associated with $\overline{\cA}$. 
Choose the nonzero scalings of the local factors so that
\begin{equation*}
 \Phi_0|_{\ker L_a}=\prod_{j=1}^h M_{a,j}\ne0
 \qquad(a\in\Lambda_0^c).
\end{equation*}
Define
\[
 J=\prod_{a\in\Lambda_0^c}L_a
 \in\F_q[X_1,X_2,X_3]_{q+2-h}.
\]
Therefore
$
 \Phi_0(\phi)\ne0
$
for any
$0\ne\phi\in\ker L_a\cap\ker L_b$ with $ a\ne b$
and $\gcd(J,\Phi_0)=1$.
Let
\begin{equation*}
 \delta=\sum_{i=1}^3X_i^q\frac{\partial}{\partial X_i},
 \qquad F_0=J\Phi_0\in\F_q[X_1,X_2,X_3]_{q+2}.
\end{equation*}
For $L=\sum_{r=1}^3 c_rX_r$ with $c_r\in\F_q$,
we have
$
 \delta L=L^q.
$
Let $(L)$ denote the principal ideal generated by $L$ in $ \F_q[X_1, X_2, X_3]$. 
As a derivation on $ \F_q[X_1, X_2, X_3]$, the operator $\delta$ strictly satisfies the Leibniz rule:
\[
 \delta(A B) = \delta(A)B + A\delta(B) \qquad \text{for all } A, B \in  \F_q[X_1, X_2, X_3].
\]
Consequently, for any element $P L \in (L)$, applying this rule alongside $\delta L = L^q$ yields
\[
 \delta(P L) = \delta(P)L + P\delta(L) = \delta(P)L + P L^q = L\bigl(\delta(P) + P L^{q-1}\bigr) \in (L).
\]
This algebraic identity proves that $\delta((L)) \subseteq (L)$, meaning $\delta$ stabilizes the ideal $(L)$. Because the ideal is preserved, $\delta$ induces a well-defined derivation on the quotient ring 
$\F_q[X_1, X_2, X_3]/(L).
$
Geometrically, this quotient ring canonically coincides with the polynomial coordinate ring of the $2$-dimensional subspace $\ker L$.
By selecting a basis $\{U_1, U_2\}$ for the dual space $(\ker L)^*$, we establish a canonical ring isomorphism for the coordinate ring:
$
\F_q[X_1, X_2, X_3]/(L) \cong \F_q[U_1, U_2].
$
Let $\widetilde{\delta}$ denote the induced derivation on $\F_q[U_1, U_2]$. Because the 
global identity $\delta(L) = L^q$ holds for any linear form $L \in \F_q[X_1, X_2, X_3]_1$, 
the action of $\widetilde{\delta}$ restricted to the degree-one generators $U_1$ and $U_2$ strictly preserves this property:
$
 \widetilde{\delta}(U_1) = U_1^q $ and $ \widetilde{\delta}(U_2) = U_2^q.
$
In the polynomial ring $\F_q[U_1, U_2]$, any derivation is uniquely determined by its evaluation on the variables $\{U_1, U_2\}$ and can be linearly expanded over the basis of partial differential operators $\{\frac{\partial}{\partial U_1}, \frac{\partial}{\partial U_2}\}$. Substituting the evaluated generators yields the explicit operator expansion:
\[
 \widetilde{\delta} 
 = \widetilde{\delta}(U_1)\frac{\partial}{\partial U_1} + \widetilde{\delta}(U_2)\frac{\partial}{\partial U_2} 
 = U_1^q\frac{\partial}{\partial U_1} + U_2^q\frac{\partial}{\partial U_2}.
\]
Moreover, for any polynomial $P \in \F_q[X_1, X_2, X_3]$, 
applying the derivation $\delta = \sum_{i=1}^3 X_i^q \partial_{X_i}$ iteratively yields
\[
 \delta^2 P 
 = \sum_{j=1}^3 X_j^q \partial_{X_j} \left( \sum_{i=1}^3 X_i^q \partial_{X_i} P \right)
 = \sum_{j=1}^3 \sum_{i=1}^3 X_j^q \left( \partial_{X_j}(X_i^q) \partial_{X_i} P + X_i^q \partial_{X_j} \partial_{X_i} P \right).
\]
This leads to 
\[
 \delta^2 P 
 = \sum_{i=1}^3 X_i^{2q} \partial_{X_i}^2 P + 2\sum_{1 \le i < j \le 3} X_i^q X_j^q 
 \partial_{X_i} \partial_{X_j} P,
\]
giving
$
 \delta^2 P = 0.
$
Recall from Equation~\eqref{eq:W-definition} that $W_{\bz} = \sum_{i \in \Lambda_0^c} z_i L_i^{q-2}$. 
After simple algebraic calculations, we have
\begin{equation*}
 W_{\bz}
 = \sum_{i\in\Lambda_0^c} \left(z_i^{1/2}L_i^{m-1}\right)^2
 = \left(\sum_{i\in\Lambda_0^c}z_i^{1/2}L_i^{m-1}\right)^2,
\end{equation*}
giving
$
 \delta W_{\bz} = 0.
$
With these foundational properties of the derivation $\delta$ in place, we now determine its action on the global product $F_0 = J\Phi_0$ and derive the resulting differential identity for the  polynomial $\Phi_0$.
\begin{lemma}\label{lem:product}
The product $F_0=J\Phi_0$ satisfies
$
 \delta F_0=0
 $
 and 
 $
 \delta\Phi_0=\Phi_0\sum_{a\in\Lambda_0^c}L_a^{q-1}.
$
\end{lemma}
\begin{proof}
Fix an index $a \in \Lambda_0^c$ and factor $J$ in $\F_q[X_1, X_2, X_3]$ as
\[
 J = L_a \prod_{b \in \Lambda_0^c \setminus \{a\}} L_b,
\]
where $\prod_{b \in \Lambda_0^c \setminus \{a\}} L_b \in \F_q[X_1, X_2, X_3]_{q+1-h}$.
Under the  ring isomorphism $\F_q[X_1, X_2, X_3]/(L_a) \cong \F_q[U_1, U_2]$ on $\ker L_a$, the $q+1-h$ restricted linear forms $\{L_b|_{\ker L_a}\}_{b \in \Lambda_0^c \setminus \{a\}}$ vanish at the $q+1-h$ distinct line intersections $[\ker L_a \cap \ker L_b]$ on $\PP(\ker L_a)$, whereas the $h$ linear forms $\{M_{a,j}\}_{j=1}^h$ vanish at the remaining $h$ rational points. Since $\Phi_0|_{\ker L_a} = \prod_{j=1}^h M_{a,j}$, the restricted product
\[
 \left.\left(\Phi_0 \prod_{b \in \Lambda_0^c \setminus \{a\}} L_b\right)\right|_{\ker L_a}
 = \left(\prod_{j=1}^h M_{a,j}\right) \left(\prod_{b \in \Lambda_0^c \setminus \{a\}} L_b|_{\ker L_a}\right) \in \F_q[U_1, U_2]_{q+1}
\]
is a product of $q+1$ pairwise nonproportional linear forms whose projective zeros sweep over all $q+1$ rational points of $\PP(\ker L_a) \cong \PP^1(\F_q)$. There exists a nonzero scalar $c_a \in \F_q^*$ such that
\[
 \left.\left(\Phi_0 \prod_{b \in \Lambda_0^c \setminus \{a\}} L_b\right)\right|_{\ker L_a} = c_a (U_1^q U_2 + U_1 U_2^q).
\]
Applying the induced derivation $\widetilde{\delta} = U_1^q \frac{\partial}{\partial U_1} + U_2^q \frac{\partial}{\partial U_2}$ on $\F_q[U_1, U_2]$, we obtain
\[
 \widetilde{\delta}\left(\left.\left(\Phi_0 \prod_{b \in \Lambda_0^c \setminus \{a\}} L_b\right)\right|_{\ker L_a}\right)
 = c_a \left( U_1^q \frac{\partial}{\partial U_1}(U_1^q U_2 + U_1 U_2^q) + U_2^q \frac{\partial}{\partial U_2}(U_1^q U_2 + U_1 U_2^q) \right)
  = 0.
\]
Because $\widetilde{\delta}$ is induced by $\delta$ modulo $(L_a)$, it follows that
\[
 \delta\left(\Phi_0 \prod_{b \in \Lambda_0^c \setminus \{a\}} L_b\right) \in (L_a),
\]
so $L_a$ divides $\delta\bigl(\Phi_0 \prod_{b \in \Lambda_0^c \setminus \{a\}} L_b\bigr)$ in $\F_q[X_1, X_2, X_3]$.
Now write $F_0 = L_a \bigl(\Phi_0 \prod_{b \in \Lambda_0^c \setminus \{a\}} L_b\bigr)$. Applying the Leibniz rule and the relation $\delta L_a = L_a^q$ yields
\[
 \delta F_0
 = (\delta L_a)\left(\Phi_0 \prod_{b \in \Lambda_0^c \setminus \{a\}} L_b\right) + L_a \delta\left(\Phi_0 \prod_{b \in \Lambda_0^c \setminus \{a\}} L_b\right)
 = L_a^q \Phi_0 \prod_{b \in \Lambda_0^c \setminus \{a\}} L_b + L_a \delta\left(\Phi_0 \prod_{b \in \Lambda_0^c \setminus \{a\}} L_b\right).
\]
We have $L_a^2 \mid L_a \delta\bigl(\Phi_0 \prod_{b \in \Lambda_0^c \setminus \{a\}} L_b\bigr)$ 
and  $L_a^2 \mid \delta F_0$ for every $a \in \Lambda_0^c$. Because the linear forms $\{L_a\}_{a \in \Lambda_0^c}$ are pairwise nonproportional, their squares $\{L_a^2\}_{a \in \Lambda_0^c}$ are pairwise coprime in the unique factorization domain $\F_q[X_1, X_2, X_3]$. Consequently,
\[
 J^2 = \prod_{a \in \Lambda_0^c} L_a^2 \;\Big|\; \delta F_0,
\]
and  there exists a homogeneous polynomial $Q \in \F_q[X_1, X_2, X_3]$ such that
$
 \delta F_0 = J^2 Q.
$
If $Q \ne 0$, comparing total degrees in $ \delta F_0 = J^2 Q$ 
via $\deg F_0 =  q+2$, $\deg \delta F_0 = 2q+1$, 
and $\deg J^2  = 2q+4-2h$ shows that $Q \in \F_q[X_1, X_2, X_3]_{\deg Q}$ has exact total degree
$
\deg Q = 2h - 3.
$
If $h = 1$, then $\deg Q  < 0$, which immediately forces $Q = 0$. Suppose now that $h \ge 2$. Since $h \le m - 1 = q/2 - 1$,   the degree $\deg Q = 2h-3$ is an odd integer satisfying $1 \le \deg Q < q$. 
Applying $\delta$ to $ \delta F_0 = J^2 Q$ and using $\delta^2 F_0 = 0$, we obtain
$
 0 = \delta^2 F_0 = \delta(J^2)Q + J^2 \delta Q = J^2 \delta Q.
$
It follows that
\begin{equation}\label{eq:delta-Q-zero}
 \delta Q = \sum_{i=1}^3 X_i^q \frac{\partial Q}{\partial X_i} = 0.
\end{equation}
Expand $Q$ in the monomial basis of $\F_q[X_1, X_2, X_3]_{\deg Q}$:
\[
 Q = \sum_{n_1+n_2+n_3=\deg Q} c_{n_1,n_2,n_3} X_1^{n_1} X_2^{n_2} X_3^{n_3} \qquad (c_{n_1,n_2,n_3} \in \F_q).
\]
Because $\deg Q < q$, for each $i \in \{1, 2, 3\}$, every nonzero monomial occurring in $X_i^q \frac{\partial Q}{\partial X_i}$ has $X_i$-degree $n_i - 1 + q \ge q$ and $X_j$-degree $n_j \le \deg Q < q$ for all $j \ne i$. Hence, for distinct indices $i \ne j$, the sets of monomials appearing in $X_i^q \frac{\partial Q}{\partial X_i}$ and $X_j^q \frac{\partial Q}{\partial X_j}$ are disjoint. Equation~\eqref{eq:delta-Q-zero} therefore requires each summand $X_i^q \frac{\partial Q}{\partial X_i}$ to vanish individually. 
We have
\[
 \frac{\partial Q}{\partial X_i} = \sum_{n_1+n_2+n_3=\deg Q} n_i c_{n_1,n_2,n_3} X_1^{n_1} \cdots X_i^{n_i-1} \cdots X_3^{n_3} = 0 \qquad (1 \le i \le 3).
\]
By the linear independence of monomials in $\F_q[X_1, X_2, X_3]$, we have 
$n_i c_{n_1,n_2,n_3} = 0$ in $\F_q$ for all $1 \le i \le 3$ and all triples $(n_1, n_2, n_3)$ 
with $n_1 + n_2 + n_3 = \deg Q$. Since $n_1 + n_2 + n_3 = 2h-3$ is odd, 
every such triple $(n_1, n_2, n_3)$ contains at least one odd integer $n_i$, 
for which $n_i = 1$ in $\F_q$. This forces $c_{n_1,n_2,n_3} = 0$ for all $(n_1, n_2, n_3)$, 
proving that $Q = 0$ and thus
$
 \delta F_0 = 0.
$

Finally, applying the Leibniz rule and the relation $\delta L_a = L_a^q$ to $J = \prod_{a \in \Lambda_0^c} L_a$ gives
\[
 \delta J = \sum_{a \in \Lambda_0^c} (\delta L_a) \prod_{b \in \Lambda_0^c \setminus \{a\}} L_b = \sum_{a \in \Lambda_0^c} L_a^q \prod_{b \in \Lambda_0^c \setminus \{a\}} L_b = J \sum_{a \in \Lambda_0^c} L_a^{q-1}.
\]
Substituting this into $0 = \delta F_0 = \delta(J \Phi_0)$ via the Leibniz rule yields
\[
 0 = (\delta J)\Phi_0 + J(\delta \Phi_0) = J\left(\delta \Phi_0 + \Phi_0 \sum_{a \in \Lambda_0^c} L_a^{q-1}\right),
\]
yielding
$
 \delta \Phi_0 = \Phi_0 \sum_{a \in \Lambda_0^c} L_a^{q-1}.
$
\end{proof}

To lift a compatible system of local binary forms defined on a planar line arrangement into a single global ternary form, we establish 
the following algebraic reconstruction lemma. The result demonstrates that such a global polynomial is uniquely determined provided its total degree is strictly less than the number of lines.

\begin{lemma}\label{lem:glue}
Let $\ell_1,\ldots,\ell_N\in\F_q[X_1,X_2,X_3]_1 \setminus \{0\}$ such that 
$\ker \ell_i \neq \ker \ell_j$ for all $1 \le i < j \le N$, and 
$\ker \ell_i \cap \ker \ell_j \cap \ker \ell_k = \{0\}$ for all $1 \le i < j < k \le N$. 
Let $0\le d<N$,
and let $g_i$ be a homogeneous binary form of degree $d$ on $\ker\ell_i$.
Suppose that $g_i$ and $g_j$ have the same restriction to
$\ker\ell_i\cap\ker\ell_j$ whenever $i\ne j$. Then there exists a unique
$g\in\F_q[X_1,X_2,X_3]_d$ such that
$
 g|_{\ker\ell_i}=g_i
$
for all $1\le i\le N.$
\end{lemma}
\begin{proof}
For each $1 \le j \le N$,  we can extend $\{\ell_j\}$ to a basis 
$\{U_{j,1}, U_{j,2}, \ell_j\}$ of $\F_q[X_1, X_2, X_3]_1$. 
Restricting to the $2$-dimensional subspace $\ker \ell_j$ annihilates $\ell_j$, 
so the restricted linear forms
\[
 \overline{U}_{j,1} = U_{j,1}|_{\ker \ell_j}, \qquad \overline{U}_{j,2} = U_{j,2}|_{\ker \ell_j}
\]
form a basis of $(\ker \ell_j)^*$, identifying the coordinate ring of $\ker \ell_j$ with $\F_q[\overline{U}_{j,1}, \overline{U}_{j,2}]$.
We construct $g$ by induction on the number of lines. For the base case $j = 1$, expand the given homogeneous binary form $g_1 \in \F_q[\overline{U}_{1,1}, \overline{U}_{1,2}]_d$ in the monomial basis of $(\ker \ell_1)^*$ as
\[
 g_1 = \sum_{m=0}^d a_m \overline{U}_{1,1}^{d-m} \overline{U}_{1,2}^m = \sum_{m=0}^d a_m \bigl(U_{1,1}|_{\ker \ell_1}\bigr)^{d-m} \bigl(U_{1,2}|_{\ker \ell_1}\bigr)^m \qquad (a_m \in \F_q).
\]
Define the homogeneous polynomial $g^{(1)} \in \F_q[X_1, X_2, X_3]_d$ by
$
 g^{(1)} = \sum_{m=0}^d a_m U_{1,1}^{d-m} U_{1,2}^m.
$
Restricting $g^{(1)}$ to $\ker \ell_1$ immediately yields
\[
 g^{(1)}|_{\ker \ell_1} = \sum_{m=0}^d a_m \bigl(U_{1,1}|_{\ker \ell_1}\bigr)^{d-m} \bigl(U_{1,2}|_{\ker \ell_1}\bigr)^m = g_1.
\]
Now let $2 \le j \le N$, and suppose inductively that we have constructed 
$g^{(j-1)} \in \F_q[X_1, X_2, X_3]_d$ satisfying
$
 g^{(j-1)}|_{\ker \ell_i} = g_i
$
for all $1 \le i < j$.
Define the residual homogeneous binary form of degree $d$ on $\ker \ell_j$ by
\[
 r_j = g_j - g^{(j-1)}|_{\ker \ell_j} \in \F_q[\overline{U}_{j,1}, \overline{U}_{j,2}]_d.
\]
For every $1 \le i < j$, 
applying the inductive hypothesis $g^{(j-1)}|_{\ker \ell_i} = g_i$ 
and the compatibility condition 
$g_j|_{\ker \ell_i \cap \ker \ell_j} = g_i|_{\ker \ell_i \cap \ker \ell_j}$ gives
$
 r_j|_{\ker \ell_i \cap \ker \ell_j} 
 = 0.
$
For each $1 \le i < j$, if $\ell_i|_{\ker \ell_j} = 0$, then $\ker \ell_j \subseteq \ker \ell_i$, which forces $\ker \ell_i = \ker \ell_j$, contradicting $\ker \ell_i \ne \ker \ell_j$. Thus $\ell_i|_{\ker \ell_j} \in (\ker \ell_j)^* \setminus \{0\}$ with $1$-dimensional kernel
$
 \ker\bigl(\ell_i|_{\ker \ell_j}\bigr) = \ker \ell_i \cap \ker \ell_j.
$
Extend $\ell_i|_{\ker \ell_j}$ to a basis $\{W_1, W_2\}$ of $(\ker \ell_j)^*$ with $W_1 = \ell_i|_{\ker \ell_j}$, and let $\{\mathbf{w}_1, \mathbf{w}_2\}$ be the dual basis of $\ker \ell_j$ satisfying $W_\alpha(\mathbf{w}_\beta) = \delta_{\alpha\beta}$, so that $\ker \ell_i \cap \ker \ell_j = \Span_{\F_q}\{\mathbf{w}_2\}$. Expanding $r_j$ in the basis $\{W_1, W_2\}$ gives
\[
 r_j = \sum_{m=0}^d c_m W_1^{d-m} W_2^m = W_1 \left(\sum_{m=0}^{d-1} c_m W_1^{d-1-m} W_2^m\right) + c_d W_2^d \qquad (c_m \in \F_q).
\]
Evaluating at $\mathbf{w}_2 \in \ker \ell_i \cap \ker \ell_j$ via 
$
 r_j|_{\ker \ell_i \cap \ker \ell_j} 
 = 0
$ yields $0 = r_j(\mathbf{w}_2) = c_d$, and hence
\[
 r_j = \bigl(\ell_i|_{\ker \ell_j}\bigr) \sum_{m=0}^{d-1} c_m W_1^{d-1-m} W_2^m,
\]
proving that $\ell_i|_{\ker \ell_j}$ divides $ r_j$ in $\F_q[\overline{U}_{j,1}, \overline{U}_{j,2}]$. 
Moreover, 
 $\ell_1|_{\ker \ell_j}, \ldots, \ell_{j-1}|_{\ker \ell_j}$ are pairwise nonproportional and hence pairwise coprime in the unique factorization domain $\F_q[\overline{U}_{j,1}, \overline{U}_{j,2}]$. Consequently,
\begin{equation}\label{eq:residual-divisibility}
 \prod_{i=1}^{j-1} \bigl(\ell_i|_{\ker \ell_j}\bigr) \;\Big|\; r_j.
\end{equation}
We now distinguish two cases according to $j - 1$ and $d$:
\begin{itemize}
\item \textbf{Case 1:} $j - 1 \le d$. By \eqref{eq:residual-divisibility}, there exists a homogeneous binary form $v_j \in \F_q[\overline{U}_{j,1}, \overline{U}_{j,2}]_{d-j+1}$ such that
\[
 r_j = \left(\prod_{i=1}^{j-1} \ell_i|_{\ker \ell_j}\right) v_j.
\]
Expanding $v_j = \sum_{m=0}^{d-j+1} b_{j,m} \overline{U}_{j,1}^{d-j+1-m} \overline{U}_{j,2}^m$ with $b_{j,m} \in \F_q$, we define its global homogeneous lift by
\[
 \widetilde{v}_j = \sum_{m=0}^{d-j+1} b_{j,m} U_{j,1}^{d-j+1-m} U_{j,2}^m \in \F_q[X_1, X_2, X_3]_{d-j+1},
\]
which satisfies $\widetilde{v}_j|_{\ker \ell_j} = v_j$, and set
$
 g^{(j)} = g^{(j-1)} + \left(\prod_{i=1}^{j-1} \ell_i\right) \widetilde{v}_j \in \F_q[X_1, X_2, X_3]_d.
$
For $1 \le i < j$,  
$
 g^{(j)}|_{\ker \ell_i} =  g_i,
$
whereas on $\ker \ell_j$, we have
\[
 g^{(j)}|_{\ker \ell_j} = g^{(j-1)}|_{\ker \ell_j} + \left(\prod_{i=1}^{j-1} \ell_i|_{\ker \ell_j}\right) \widetilde{v}_j|_{\ker \ell_j} = g^{(j-1)}|_{\ker \ell_j} + r_j = g_j.
\]

\item \textbf{Case 2:} $j - 1 > d$. If $r_j \ne 0$, comparing degrees in \eqref{eq:residual-divisibility} gives
$
 j - 1 \le d,
$
contradicting $j - 1 > d$. Hence $r_j = 0$, so $g^{(j-1)}|_{\ker \ell_j} = g_j$, and setting $g^{(j)} = g^{(j-1)}$ gives $g^{(j)}|_{\ker \ell_i} = g_i$ for all $1 \le i \le j$.
\end{itemize}
By induction, $g = g^{(N)} \in \F_q[X_1, X_2, X_3]_d$ satisfies $g|_{\ker \ell_i} = g_i$ for all $1 \le i \le N$.

To prove uniqueness, suppose $g, \widetilde{g} \in \F_q[X_1, X_2, X_3]_d$ both satisfy $g|_{\ker \ell_i} = \widetilde{g}|_{\ker \ell_i} = g_i$ for all $1 \le i \le N$. For each $1 \le i \le N$, expanding $g - \widetilde{g}$ in the basis $\{U_{i,1}, U_{i,2}, \ell_i\}$ of $\F_q[X_1, X_2, X_3]_1$ gives
$
 g - \widetilde{g} = P_i(U_{i,1}, U_{i,2}) + \ell_i Q_i
$
for some $P_i \in \F_q[U_{i,1}, U_{i,2}]_d$ and $Q_i \in \F_q[X_1, X_2, X_3]_{d-1}$. Restricting to $\ker \ell_i$ yields
\[
 0 = (g - \widetilde{g})|_{\ker \ell_i} = P_i(\overline{U}_{i,1}, \overline{U}_{i,2}).
\]
Since $\overline{U}_{i,1}, \overline{U}_{i,2}$ are linearly independent, $P_i = 0$ and thus $g - \widetilde{g} = \ell_i Q_i$, meaning $\ell_i \mid (g - \widetilde{g})$ in $\F_q[X_1, X_2, X_3]$.
This leads to
$
 \prod_{i=1}^N \ell_i \;\Big|\; (g - \widetilde{g}).
$
If $g - \widetilde{g} \ne 0$, then $N = \deg\bigl(\prod_{i=1}^N \ell_i\bigr) \le \deg(g - \widetilde{g}) = d$, contradicting $d < N$. Therefore $g = \widetilde{g}$.
\end{proof}

With the algebraic reconstruction established, we now bound the dimension of the orthogonal complement for the sum of the local subcodes. 
This upper bound is governed by the dimension of a specific homogeneous polynomial kernel.
Define
\begin{equation}\label{eq:critical-kernel}
 \kappa=\dim_{\F_q}
 \bigl\{T\in\F_q[X_1,X_2,X_3]_{h-2}:JT\in(X_1^m,X_2^m,X_3^m)\bigr\}.
\end{equation}
\begin{proposition}\label{prop:colon}
We have
$
 \dim_{\F_q}\left(\sum_{a\in\Lambda_0^c}
          \cD_{\Lambda_0\cup\{a\}}\right)^\perp\le3+\kappa,
$
where the orthogonal complement is taken in $\F_q^{\Lambda_0^c}$.
\end{proposition}
\begin{proof}
Let $\cZ = \bigcap_{a\in\Lambda_0^c}\cD_{\Lambda_0\cup\{a\}}^\perp$. Then
$\cZ$ can be viewed as a subspace of $\F_q^{\Lambda_0^c}$.
Fix $\bz\in\cZ $. For each $a\in\Lambda_0^c$,
$
 \Phi_0|_{\ker L_a}=\prod_{j=1}^hM_{a,j}.
$
By Lemma~\ref{lem:coordinates}, there exist unique scalars $\mu_{a,1},\ldots,\mu_{a,h}\in\F_q$ such that
\begin{equation}\label{eq:local-annihilator-powers}
 W_{\bz}|_{\ker L_a}=\sum_{j=1}^h\mu_{a,j}M_{a,j}^{q-2}.
\end{equation}
On each subspace $\ker L_a$, define the homogeneous binary form of degree $h-1$ by
\[
 \Psi_{\bz,a}=\sum_{j=1}^h\mu_{a,j}
 \prod_{\substack{1\le\ell\le h\\\ell\ne j}}M_{a,\ell}.
\]
For any distinct $a,b\in\Lambda_0^c$, choose a nonzero vector $\phi_{ab}\in\ker L_a\cap\ker L_b$. 
We have $M_{a,j}(\phi_{ab})\ne0$ for all $1\le j\le h$. Using $u^{-1}=u^{q-2}$ for all $u\in\F_q^*$ 
alongside \eqref{eq:local-annihilator-powers}, we obtain
\begin{align*}
 \Psi_{\bz,a}(\phi_{ab})
 &=\sum_{j=1}^h\mu_{a,j}\frac{\Phi_0(\phi_{ab})}{M_{a,j}(\phi_{ab})}
  =\Phi_0(\phi_{ab})\sum_{j=1}^h\mu_{a,j}M_{a,j}(\phi_{ab})^{q-2}\\
 &=\Phi_0(\phi_{ab})W_{\bz}(\phi_{ab})
  =\Psi_{\bz,b}(\phi_{ab}),
\end{align*}
giving
 $\Psi_{\bz,a}|_{\ker L_a\cap\ker L_b} = \Psi_{\bz,b}|_{\ker L_a\cap\ker L_b}$. 
Because $0\le h-1 < q+2-h = |\Lambda_0^c|$, Lemma~\ref{lem:glue} provides a unique homogeneous 
polynomial $\Psi_{\bz}\in\F_q[X_1,X_2,X_3]_{h-1}$ satisfying
$
 \Psi_{\bz}|_{\ker L_a}=\Psi_{\bz,a}$
 for any 
$a\in\Lambda_0^c$.
For each $a\in\Lambda_0^c$, restricting the product $\Phi_0\Psi_{\bz}$ to $\ker L_a$ gives
\[
 (\Phi_0\Psi_{\bz})\big|_{\ker L_a}
 = \left(\prod_{\ell=1}^h M_{a,\ell}\right)\left(\sum_{j=1}^h\mu_{a,j}\prod_{\substack{1\le\ell\le h\\\ell\ne j}}M_{a,\ell}\right)
 = \sum_{j=1}^h\mu_{a,j}M_{a,j}\prod_{\substack{1\le\ell\le h\\\ell\ne j}}M_{a,\ell}^2.
\]
Let $\widetilde{\delta}$ be the derivation on the coordinate ring of $\ker L_a$ induced by $\delta$. 
Since $\widetilde{\delta}(M_{a,j}) = M_{a,j}^q$ and $\widetilde{\delta}(M_{a,\ell}^2) = 0$, applying $\widetilde{\delta}$ yields
\[
\begin{aligned}
 \delta(\Phi_0\Psi_{\bz})\big|_{\ker L_a}
 &= \widetilde{\delta}\bigl((\Phi_0\Psi_{\bz})|_{\ker L_a}\bigr)
  = \sum_{j=1}^h\mu_{a,j}M_{a,j}^q\prod_{\substack{1\le\ell\le h\\\ell\ne j}}M_{a,\ell}^2\\
 &= \left(\prod_{\ell=1}^h M_{a,\ell}\right)^2 \sum_{j=1}^h\mu_{a,j}M_{a,j}^{q-2}
  = (\Phi_0^2W_{\bz})\big|_{\ker L_a}.
\end{aligned}
\]
This implies  $L_a\mid\bigl(\Phi_0^2W_{\bz}+\delta(\Phi_0\Psi_{\bz})\bigr)$ for all $a\in\Lambda_0^c$. 
Since the linear forms $\{L_a\}_{a\in\Lambda_0^c}$ are pairwise coprime in $\F_q[X_1,X_2,X_3]$, 
their product $J = \prod_{a\in\Lambda_0^c} L_a$ divides $\Phi_0^2W_{\bz}+\delta(\Phi_0\Psi_{\bz})$. 
Moreover, by the Leibniz rule and Lemma~\ref{lem:product},
\[
 \Phi_0^2W_{\bz}+\delta(\Phi_0\Psi_{\bz})
 = \Phi_0^2W_{\bz} + (\delta\Phi_0)\Psi_{\bz} + \Phi_0\delta\Psi_{\bz}
 = \Phi_0\left(\Phi_0W_{\bz} + \Psi_{\bz}\sum_{a\in\Lambda_0^c}L_a^{q-1} + \delta\Psi_{\bz}\right),
\]
so $\Phi_0$ also divides $\Phi_0^2W_{\bz}+\delta(\Phi_0\Psi_{\bz})$. Since $\gcd(J,\Phi_0)=1$ and $F_0=J\Phi_0$, it follows that 
$F_0\mid\bigl(\Phi_0^2W_{\bz}+\delta(\Phi_0\Psi_{\bz})\bigr)$. Comparing total degrees via 
$\deg(\Phi_0^2W_{\bz}+\delta(\Phi_0\Psi_{\bz})) = 2h+q-2$ and $\deg F_0 = q+2$, we obtain a unique 
homogeneous polynomial $Q_{\bz}\in\F_q[X_1,X_2,X_3]_{2h-4}$ such that
\begin{equation}\label{eq:annihilator-quotient}
 \Phi_0^2W_{\bz}+\delta(\Phi_0\Psi_{\bz})
 =F_0Q_{\bz}.
\end{equation}
If $h=1$, then $2h-4  < 0$, which forces $\Phi_0^2W_{\bz}+\delta(\Phi_0\Psi_{\bz}) = 0$, and we set $Q_{\bz}=T_{\bz}=0$.
Suppose now that $h\ge2$. Applying $\delta$ to \eqref{eq:annihilator-quotient} and using $\delta(\Phi_0^2)=2\Phi_0\delta\Phi_0=0$, $\delta W_{\bz}=0$, $\delta^2=0$, and $\delta F_0=0$ gives
$
 F_0\delta Q_{\bz} = 0.
$
Since $F_0\ne0$, canceling $F_0$ yields
\begin{equation}\label{eq:delta-Qz-zero}
 \delta Q_{\bz} = \sum_{i=1}^3 X_i^q\frac{\partial Q_{\bz}}{\partial X_i} = 0.
\end{equation}
Expand $Q_{\bz}$ in the monomial basis of $\F_q[X_1,X_2,X_3]_{2h-4}$:
\[
 Q_{\bz} = \sum_{n_1+n_2+n_3=2h-4} c_{n_1,n_2,n_3} X_1^{n_1}X_2^{n_2}X_3^{n_3} \qquad (c_{n_1,n_2,n_3}\in\F_q).
\]
Because $h\le m-1 = q/2-1$, we have $0\le 2h-4< q$. Thus, for each $i\in\{1,2,3\}$, every nonzero monomial occurring in 
$X_i^q\frac{\partial Q_{\bz}}{\partial X_i}$ has $X_i$-degree $n_i-1+q\ge q$ and $X_j$-degree $n_j\le 2h-4 < q$ for all $j\ne i$. 
Consequently, the sets of monomials appearing in $X_i^q\frac{\partial Q_{\bz}}{\partial X_i}$ for $i=1,2,3$ are pairwise disjoint, 
so \eqref{eq:delta-Qz-zero} requires each summand to vanish individually:
\[
 \frac{\partial Q_{\bz}}{\partial X_i} = \sum_{n_1+n_2+n_3=2h-4} n_i c_{n_1,n_2,n_3} X_1^{n_1}\cdots X_i^{n_i-1}\cdots X_3^{n_3} = 0 \qquad (1\le i\le 3).
\]
Hence $n_i c_{n_1,n_2,n_3} = 0$ in $\F_q$ for all $1\le i\le 3$. In characteristic two, whenever $c_{n_1,n_2,n_3}\ne0$, the exponents $n_1,n_2,n_3$ must all be even, 
say $(n_1,n_2,n_3)=(2m_1,2m_2,2m_3)$ with $m_1+m_2+m_3=h-2$. Setting
\[
 T_{\bz} = \sum_{m_1+m_2+m_3=h-2} c_{2m_1,2m_2,2m_3}^{1/2} X_1^{m_1}X_2^{m_2}X_3^{m_3} \in \F_q[X_1,X_2,X_3]_{h-2}
\]
gives the unique homogeneous polynomial satisfying
$
 Q_{\bz} = T_{\bz}^2,
$
which also holds for $h=1$ under the convention $Q_{\bz}=T_{\bz}=0$.
Next, define the homogeneous polynomial of degree $q+1$ by
\[
 K_{\bz}=\sum_{a\in\Lambda_0^c}z_a\frac{F_0}{L_a}+J\Psi_{\bz}\in\F_q[X_1,X_2,X_3]_{q+1}.
\]
For each $a\in\Lambda_0^c$, applying $\delta$ to $F_0 = L_a(F_0/L_a)$ via $\delta F_0=0$ 
and $\delta L_a = L_a^q$ gives
\[
 0 = \delta F_0 = (\delta L_a)\frac{F_0}{L_a} + L_a\delta\left(\frac{F_0}{L_a}\right) = L_a\left(F_0L_a^{q-2} + \delta\left(\frac{F_0}{L_a}\right)\right),
\]
so $\delta(F_0/L_a) = F_0L_a^{q-2}$. Using $\delta J = J\sum_{a\in\Lambda_0^c}L_a^{q-1}$ from the proof of 
Lemma~\ref{lem:product} and $F_0=J\Phi_0$, we compute
\begin{align*}
 \delta K_{\bz}
 &=\sum_{a\in\Lambda_0^c}z_aF_0L_a^{q-2}+(\delta J)\Psi_{\bz}+J\delta\Psi_{\bz}
  =F_0W_{\bz}+(\delta J)\Psi_{\bz}+J\delta\Psi_{\bz}\\
 &=J\left(\Phi_0W_{\bz}
     +\Psi_{\bz}\sum_{a\in\Lambda_0^c}L_a^{q-1}
     +\delta\Psi_{\bz}\right).
\end{align*}
Multiplying by $\Phi_0$ and substituting $\delta\Phi_0=\Phi_0\sum_{a\in\Lambda_0^c}L_a^{q-1}$ from Lemma~\ref{lem:product} 
alongside \eqref{eq:annihilator-quotient} yields
\begin{align*}
 \Phi_0\delta K_{\bz}
 &=J\bigl(\Phi_0^2W_{\bz}
          +(\delta\Phi_0)\Psi_{\bz}+\Phi_0\delta\Psi_{\bz}\bigr)\\
 &=J\bigl(\Phi_0^2W_{\bz}+\delta(\Phi_0\Psi_{\bz})\bigr)
  =\Phi_0J^2T_{\bz}^2,
\end{align*}
giving
$
 \delta K_{\bz}=J^2T_{\bz}^2.
$
Since $\deg K_{\bz} = q+1 = 2m+1$ is odd, for every monomial $X_1^{d_1}X_2^{d_2}X_3^{d_3}$ of degree $2m+1$, the sum $d_1+d_2+d_3$ is 
odd, so the exponent parity triple $(d_1\bmod 2, d_2\bmod 2, d_3\bmod 2)$ must belong to $\{(1,0,0), (0,1,0), (0,0,1), (1,1,1)\}$. 
Grouping monomials of $K_{\bz}$ into these four parity classes and factoring out $X_1, X_2, X_3$, or $X_1X_2X_3$ leaves purely 
even exponents; taking square roots of the coefficients yields a unique decomposition
\begin{equation}\label{eq:Kz-parity-decomp}
 K_{\bz}=X_1A_1^2+X_2A_2^2+X_3A_3^2+X_1X_2X_3B^2,
\end{equation}
where $A_1,A_2,A_3\in\F_q[X_1,X_2,X_3]_m$ and $B\in\F_q[X_1,X_2,X_3]_{m-1}$. 
Applying $\delta = \sum_{i=1}^3 X_i^q\frac{\partial}{\partial X_i}$ to \eqref{eq:Kz-parity-decomp} gives
\[
 \delta K_{\bz}=\sum_{i=1}^3X_i^qA_i^2+
 (X_1^qX_2X_3+X_1X_2^qX_3+X_1X_2X_3^q)B^2.
\]
Because $q=2m$ is even, every monomial occurring in $J^2T_{\bz}^2$ and in $\sum_{i=1}^3 X_i^qA_i^2 = \sum_{i=1}^3 (X_i^m A_i)^2$ 
has exponent parity triple $(0,0,0)$, whereas the monomials in $X_1^qX_2X_3B^2$, $X_1X_2^qX_3B^2$, and $X_1X_2X_3^qB^2$ 
have exponent parity triples $(0,1,1)$, $(1,0,1)$, and $(1,1,0)$, respectively. Since these four parity classes of monomials are pairwise disjoint, 
equating $\delta K_{\bz} = J^2T_{\bz}^2$ forces $X_1^qX_2X_3B^2 = 0$, hence $B = 0$, and
\[
 (JT_{\bz})^2 = \sum_{i=1}^3 X_i^{2m}A_i^2 = (X_1^mA_1 + X_2^mA_2 + X_3^mA_3)^2.
\]
Taking the unique square root in $\F_q[X_1,X_2,X_3]$ yields
\begin{equation}\label{eq:colon-membership}
 JT_{\bz}=X_1^mA_1+X_2^mA_2+X_3^mA_3\in(X_1^m,X_2^m,X_3^m).
\end{equation}
Thus $T_{\bz}$ belongs to the subspace $\bigl\{T\in\F_q[X_1,X_2,X_3]_{h-2}:JT\in(X_1^m,X_2^m,X_3^m)\bigr\}$ of dimension $\kappa$.

By the linear independence of $M_{a,1}^{q-2},\ldots,M_{a,h}^{q-2}$ in \eqref{eq:local-annihilator-powers}, each coefficient $\mu_{a,j}$ is 
an $\F_q$-linear function of $\bz\in\cZ$. Consequently, $\bz\mapsto\Psi_{\bz,a}$ is $\F_q$-linear, and the uniqueness in Lemma~\ref{lem:glue} 
ensures that $\bz\mapsto\Psi_{\bz}$ is $\F_q$-linear. It follows from \eqref{eq:annihilator-quotient} and the definition of $K_{\bz}$ that both 
$\bz\mapsto Q_{\bz}$ and $\bz\mapsto K_{\bz}$ are $\F_q$-linear. 
Since $Q_{\bz}=T_{\bz}^2$, for any $\lambda\in\F_q$ and $\bz,\bw\in\cZ$ we have
$
 T_{\lambda\bz+\bw}=\lambda^{1/2}T_{\bz}+T_{\bw}.
$
Thus the map $\Theta:\cZ\to\F_q[X_1,X_2,X_3]_{h-2}$ given by $\Theta(\bz)=T_{\bz}$ is inverse-Frobenius-semilinear, and its image 
$\operatorname{im}\Theta$ is an $\F_q$-subspace of dimension at most $\kappa$.

To bound $\dim_{\F_q}\ker\Theta$, suppose $\bz\in\ker\Theta$, so that $T_{\bz}=0$. Then \eqref{eq:colon-membership} reduces to
\begin{equation}\label{eq:syzygy-relation}
 X_1^mA_1+X_2^mA_2+X_3^mA_3=0,
 \qquad A_1,A_2,A_3\in\F_q[X_1,X_2,X_3]_m.
\end{equation}
A monomial of total degree $2m$ can be divisible by both $X_i^m$ and $X_j^m$ for $1\le i<j\le 3$ if and only if it equals $X_i^mX_j^m$. 
Hence, any monomial in $A_i$ other than $X_j^m$ ($j\ne i$) produces a monomial in $X_i^mA_i$ that cannot appear in the other two summands 
of \eqref{eq:syzygy-relation} and must therefore have coefficient zero. Thus there exist scalars $\alpha_1,\alpha_2,\beta_1,\beta_2,\gamma_1,\gamma_2\in\F_q$ such that
\[
 A_1=\alpha_1 X_2^m+\beta_1 X_3^m,\qquad
 A_2=\alpha_2 X_1^m+\gamma_1 X_3^m,\qquad
 A_3=\beta_2 X_1^m+\gamma_2 X_2^m.
\]
Substituting these expressions into \eqref{eq:syzygy-relation} gives
\[
 (\alpha_1+\alpha_2)X_1^mX_2^m + (\beta_1+\beta_2)X_1^mX_3^m + (\gamma_1+\gamma_2)X_2^mX_3^m = 0.
\]
By the linear independence of monomials and characteristic two, we obtain $\alpha_1=\alpha_2=\alpha$, $\beta_1=\beta_2=\beta$, and 
$\gamma_1=\gamma_2=\gamma$ for some $\alpha,\beta,\gamma\in\F_q$. Substituting $A_1,A_2,A_3$ and $B=0$ into \eqref{eq:Kz-parity-decomp} and using $2m=q$ yields
\[
\begin{aligned}
 K_{\bz}
 &= X_1(\alpha^2 X_2^q + \beta^2 X_3^q) + X_2(\alpha^2 X_1^q + \gamma^2 X_3^q) + X_3(\beta^2 X_1^q + \gamma^2 X_2^q)\\
 &= \alpha^2(X_1^qX_2+X_1X_2^q)
   +\beta^2(X_1^qX_3+X_1X_3^q)
   +\gamma^2(X_2^qX_3+X_2X_3^q).
\end{aligned}
\]
Thus, for all $\bz\in\ker\Theta$, the polynomial $K_{\bz}$ lies in the subspace 
$\Span_{\F_q}\{X_1^qX_2+X_1X_2^q,\, X_1^qX_3+X_1X_3^q,\, X_2^qX_3+X_2X_3^q\}$ of dimension at most $3$.

Finally, we verify that the $\F_q$-linear map $\bz\mapsto K_{\bz}$ is injective on $\cZ$. Suppose $K_{\bz}=0$. Then
\[
 \sum_{b\in\Lambda_0^c}z_b\Phi_0\prod_{c\in\Lambda_0^c\setminus\{b\}}L_c + \left(\prod_{c\in\Lambda_0^c}L_c\right)\Psi_{\bz} = 0.
\]
For any fixed $a\in\Lambda_0^c$, restricting this identity to $\ker L_a$ annihilates the term $\bigl(\prod_{c\in\Lambda_0^c}L_c\bigr)\Psi_{\bz}$ 
and every summand with $b\ne a$ (since $L_a|_{\ker L_a}=0$), leaving
\[
 z_a(\Phi_0|_{\ker L_a})\prod_{c\in\Lambda_0^c\setminus\{a\}}(L_c|_{\ker L_a})=0.
\]
Because $\Phi_0|_{\ker L_a}\ne0$ and $L_c|_{\ker L_a}\ne0$ for all $c\ne a$, their product in the integral domain of polynomial 
forms on $\ker L_a$ is nonzero, which forces $z_a=0$ for every $a\in\Lambda_0^c$, so $\bz=\mathbf{0}$. 
Consequently, the injectivity of $\bz\mapsto K_{\bz}$ implies $\dim_{\F_q}\ker\Theta\le3$. 
Applying the rank-nullity theorem to the inverse-Frobenius-semilinear map $\Theta$ gives
\[
 \dim_{\F_q}\cZ
 =\dim_{\F_q}\ker\Theta
   +\dim_{\F_q}\operatorname{im}\Theta
 \le3+\kappa,
\]
completing the proof.
\end{proof}

\section{Dimension bounds and proof of the main theorem}
To evaluate the critical kernel $\kappa$, we first establish that multiplication by $J$ is injective on low-degree forms modulo the ideal $(X_1^m, X_2^m, X_3^m)$. Recall that $\deg J=q+2-h$. The following lemma relies only on the rational lines $\ker L_a$ for $a\in\Lambda_0^c$ and the non-concurrency condition $\ker L_a \cap \ker L_b \cap \ker L_c = \{\mathbf{0}\}$ for distinct $a,b,c\in\Lambda_0^c$.

\begin{lemma}\label{lem:inject}
Let $0\le r\le h-3$, and let $Q\in\F_q[X_1,X_2,X_3]_r$. If 
$
 JQ\in(X_1^m,X_2^m,X_3^m)$
 then $Q=0.
$
\end{lemma}
\begin{proof}
Since $JQ\in(X_1^m,X_2^m,X_3^m)$,
there are homogeneous polynomials $A_1,A_2,A_3\in\F_q[X_1,X_2,X_3]_d$ such that
\begin{equation}\label{eq:inject-relation}
 JQ=X_1^mA_1+X_2^mA_2+X_3^mA_3,
\end{equation}
where the degree $d$ satisfies
$
 d=(q+2-h)+r-m=m+2-h+r\le m-1.
$
Define the homogeneous polynomial
\[
 K=X_1A_1^2+X_2A_2^2+X_3A_3^2\in\F_q[X_1,X_2,X_3]_{2d+1}.
\]
The total degree of $K$ is the odd integer
$
 \deg K = 2d+1 =  q+5-2h+2r \le  q-1.
$
For each $1\le i\le 3$,  differentiating $K$ with respect to $X_i$ gives
\[
 \frac{\partial K}{\partial X_i} = \sum_{j=1}^3\left(\frac{\partial X_j}{\partial X_i}A_j^2 + X_j\frac{\partial(A_j^2)}{\partial X_i}\right) = A_i^2.
\]
Consequently, applying $\delta = \sum_{i=1}^3 X_i^q\frac{\partial}{\partial X_i}$ and using $q=2m$ 
together with \eqref{eq:inject-relation} yields
\begin{equation}\label{eq:inject-derivatives}
 \delta K=\sum_{i=1}^3X_i^qA_i^2=\left(\sum_{i=1}^3X_i^mA_i\right)^2=(JQ)^2.
\end{equation}

Fix an index $a\in\Lambda_0^c$. Since $L_a\mid J$, restricting $\delta K = (JQ)^2$ to $\ker L_a$ gives $(\delta K)|_{\ker L_a} = 0$. 
Under the coordinate ring isomorphism $\F_q[X_1,X_2,X_3]/(L_a) \cong \F_q[U_1,U_2]$ on $\ker L_a$, let 
$\overline{K}_a = K|_{\ker L_a} \in \F_q[U_1,U_2]_{2d+1}$. The induced derivation 
$\widetilde{\delta} = U_1^q\frac{\partial}{\partial U_1} + U_2^q\frac{\partial}{\partial U_2}$ satisfies
\begin{equation}\label{eq:delta-Ka-zero}
 \widetilde{\delta}(\overline{K}_a) = U_1^q\frac{\partial\overline{K}_a}{\partial U_1} + U_2^q\frac{\partial\overline{K}_a}{\partial U_2} = (\delta K)|_{\ker L_a} = 0.
\end{equation}
Expand $\overline{K}_a$ in the monomial basis of $\F_q[U_1,U_2]_{2d+1}$:
\[
 \overline{K}_a = \sum_{n_1+n_2=2d+1} c_{n_1,n_2} U_1^{n_1}U_2^{n_2} \qquad (c_{n_1,n_2}\in\F_q).
\]
Because $n_1+n_2 = 2d+1 < q$, every nonzero monomial occurring in $U_1^q\frac{\partial\overline{K}_a}{\partial U_1}$ has $U_1$-degree 
$n_1-1+q\ge q$ and $U_2$-degree $n_2 < q$, whereas every nonzero monomial in $U_2^q\frac{\partial\overline{K}_a}{\partial U_2}$ 
has $U_1$-degree $n_1 < q$ and $U_2$-degree $n_2-1+q\ge q$. Since these two sets of monomials are disjoint, \eqref{eq:delta-Ka-zero} 
forces each summand to vanish individually:
\[
 \frac{\partial\overline{K}_a}{\partial U_1} = \sum_{n_1+n_2=2d+1} n_1 c_{n_1,n_2} U_1^{n_1-1}U_2^{n_2} = 0, \qquad
 \frac{\partial\overline{K}_a}{\partial U_2} = \sum_{n_1+n_2=2d+1} n_2 c_{n_1,n_2} U_1^{n_1}U_2^{n_2-1} = 0.
\]
By the linear independence of monomials, $n_1 c_{n_1,n_2} = n_2 c_{n_1,n_2} = 0$ in $\F_q$ for all pairs $(n_1,n_2)$ with $n_1+n_2 = 2d+1$. Since $2d+1$ is odd, 
every such pair contains at least one odd integer $n_i$, for which $n_i = 1$ in $\F_q$, forcing $c_{n_1,n_2} = 0$. Thus $\overline{K}_a = K|_{\ker L_a} = 0$, 
which means $L_a\mid K$ for every $a\in\Lambda_0^c$. 
It follows that there exists a polynomial $B\in\F_q[X_1,X_2,X_3]$ such that
$
 K=JB.
$
For any distinct $a,b\in\Lambda_0^c$, choose a nonzero vector $\phi_{ab}\in\ker L_a\cap\ker L_b$, and factor $J$ as
\[
 J = L_a L_b J_{ab}, \qquad \text{where } J_{ab} = \prod_{c\in\Lambda_0^c\setminus\{a,b\}} L_c.
\]
Since $L_a(\phi_{ab}) = L_b(\phi_{ab}) = 0$, we have $J(\phi_{ab}) = 0$. Furthermore, for each $1\le i\le 3$, applying the Leibniz rule to $J = L_a L_b J_{ab}$ gives
\[
 \frac{\partial J}{\partial X_i} = \frac{\partial L_a}{\partial X_i} L_b J_{ab} + L_a \frac{\partial L_b}{\partial X_i} J_{ab} + L_a L_b \frac{\partial J_{ab}}{\partial X_i}.
\]
Evaluating at $\phi_{ab}$ yields $\frac{\partial J}{\partial X_i}(\phi_{ab}) = 0$ for all $1\le i\le 3$. Differentiating
$ K=JB$ with respect to $X_i$ and combining with $\frac{\partial K}{\partial X_i} = A_i^2$ at $\phi_{ab}$ gives
\[
 A_i(\phi_{ab})^2
 =\frac{\partial K}{\partial X_i}(\phi_{ab})
 =\frac{\partial J}{\partial X_i}(\phi_{ab})B(\phi_{ab}) + J(\phi_{ab})\frac{\partial B}{\partial X_i}(\phi_{ab})
 =0 \qquad (1\le i\le 3),
\]
and hence $A_i(\phi_{ab}) = 0$ for every $1\le i\le 3$ and all distinct $a,b\in\Lambda_0^c$.

Finally, fix $i\in\{1,2,3\}$ and $a\in\Lambda_0^c$. For each $b\in\Lambda_0^c\setminus\{a\}$, the restricted linear form $L_b|_{\ker L_a}$ on $\ker L_a$ 
has kernel $\ker L_a\cap\ker L_b = \Span_{\F_q}\{\phi_{ab}\}$. Since $A_i(\phi_{ab}) = 0$, the form $A_i|_{\ker L_a}$ vanishes on 
$\ker(L_b|_{\ker L_a})$, so $L_b|_{\ker L_a}$ divides $A_i|_{\ker L_a}$. By the non-concurrency condition 
$\ker L_a\cap\ker L_b\cap\ker L_c = \{\mathbf{0}\}$, the $q+1-h$ linear forms $\{L_b|_{\ker L_a}\}_{b\in\Lambda_0^c\setminus\{a\}}$ 
are pairwise nonproportional and hence pairwise coprime, which implies
\[
 \prod_{b\in\Lambda_0^c\setminus\{a\}}\bigl(L_b|_{\ker L_a}\bigr) \;\Big|\; A_i|_{\ker L_a}.
\]
Since $h\le m-1$ and $q=2m$, the degree $d = \deg A_i$ satisfies
\[
 d \le m-1 < 2m+1-h = q+1-h = \deg\left(\prod_{b\in\Lambda_0^c\setminus\{a\}} L_b|_{\ker L_a}\right).
\]
This forces $A_i|_{\ker L_a} = 0$, so $L_a\mid A_i$ for every $a\in\Lambda_0^c$, and consequently $J = \prod_{a\in\Lambda_0^c} L_a$ divides $A_i$. 
Because $\deg A_i = d < q+2-h = \deg J$, we obtain $A_1 = A_2 = A_3 = 0$, yielding $Q = 0$.
\end{proof}

We now utilize the injectivity established in Lemma~\ref{lem:inject} to analyze the colon ideal associated with $J$. By demonstrating that its corresponding quotient ring forms a standard graded Artinian Gorenstein algebra, we can apply the unimodality of its Hilbert function to rigorously bound the critical kernel $\kappa$.
\begin{proposition}\label{prop:kappa}
The critical kernel in \eqref{eq:critical-kernel} satisfies
$
 \kappa\le h-1.
$
\end{proposition}
\begin{proof}
Let $\mathfrak{q}=(X_1^m,X_2^m,X_3^m)\subseteq\F_q[X_1,X_2,X_3]$, 
so that $\kappa=\dim_{\F_q}\{T\in\F_q[X_1,X_2,X_3]_{h-2}:JT\in\mathfrak{q}\}$. 
If $h=1$, then $\F_q[X_1,X_2,X_3]_{h-2}=\{0\}$ and $\kappa=0$; if $h=2$, then $\F_q[X_1,X_2,X_3]_{h-2}=\F_q$ has dimension $1$, so $\kappa\le 1$.
Assume henceforth that $h\ge3$.
For every homogeneous polynomial 
\[
P = \sum_{n_1+n_2+n_3=3(m-1)} c_{n_1,n_2,n_3} X_1^{n_1}X_2^{n_2}X_3^{n_3} \in \F_q[X_1,X_2,X_3]_{3(m-1)},
\]
define the $\F_q$-linear functional $\mathcal{L}:\F_q[X_1,X_2,X_3]_{3(m-1)}\to\F_q$ 
by extracting the coefficient of $X_1^{m-1}X_2^{m-1}X_3^{m-1}$:
$
\mathcal{L}(P)
 =[X_1^{m-1}X_2^{m-1}X_3^{m-1}]P
 =c_{m-1,m-1,m-1}.
$
For each $d\ge0$, let $\mathfrak{q}_d=\mathfrak{q}\cap\F_q[X_1,X_2,X_3]_d$ denote the degree-$d$ homogeneous component of 
$\mathfrak{q}$. Every monomial $X_1^{n_1}X_2^{n_2}X_3^{n_3}$ occurring in $\mathfrak{q}$ is divisible by $X_i^m$ 
for at least one $i\in\{1,2,3\}$, meaning such $n_i\ge m$. Consequently, 
$
\mathcal{L}(\mathfrak{q}_{3(m-1)})=\{0\}.
$
Moreover, if $d>3(m-1)$, then $
 \F_q[X_1,X_2,X_3]_d = \mathfrak{q}_d.
 $
For each $0\le d\le 3(m-1)$, the residue classes of the monomials in
\[
 \mathcal{M}_d = \bigl\{X_1^{u_1}X_2^{u_2}X_3^{u_3} : 0\le u_1,u_2,u_3\le m-1,\ u_1+u_2+u_3=d\bigr\}
\]
form an $\F_q$-basis of the quotient space $\F_q[X_1,X_2,X_3]_d/\mathfrak{q}_d$. Consider the bilinear pairing 
\[
\bigl(\F_q[X_1,X_2,X_3]_d/\mathfrak{q}_d\bigr)\times\bigl(\F_q[X_1,X_2,X_3]_{3(m-1)-d}/\mathfrak{q}_{3(m-1)-d}\bigr)\longrightarrow\F_q,
\]
\begin{equation}\label{eq:ambient-perfect-pairing}
 (u+\mathfrak{q}_d,\,v+\mathfrak{q}_{3(m-1)-d})\longmapsto \mathcal{L}(uv).
\end{equation}
If $u'=u+q_1$ with $q_1\in\mathfrak{q}_d$ and $v'=v+q_2$ with $q_2\in\mathfrak{q}_{3(m-1)-d}$, then $u'v'-uv = q_1v + uq_2 + q_1q_2 \in \mathfrak{q}_{3(m-1)}$,  
showing that $\mathcal{L}(u'v')=\mathcal{L}(uv)$
and 
$ \eqref{eq:ambient-perfect-pairing}$ is well-defined. Furthermore, for any basis monomials 
$X_1^{u_1}X_2^{u_2}X_3^{u_3}\in\mathcal{M}_d$ and $X_1^{v_1}X_2^{v_2}X_3^{v_3}\in\mathcal{M}_{3(m-1)-d}$, we have
\[
\mathcal{L}(X_1^{u_1+v_1}X_2^{u_2+v_2}X_3^{u_3+v_3})
 =\begin{cases}
 1,&\text{if }(v_1,v_2,v_3)=(m-1-u_1,\,m-1-u_2,\,m-1-u_3),\\
 0,&\text{otherwise}.
 \end{cases}
\]
Because $0\le u_i\le m-1$ if and only if $0\le m-1-u_i\le m-1$, each basis element $X_1^{u_1}X_2^{u_2}X_3^{u_3}+\mathfrak{q}_d$ pairs to $1$ with the unique complementary basis element $X_1^{m-1-u_1}X_2^{m-1-u_2}X_3^{m-1-u_3}+\mathfrak{q}_{3(m-1)-d}$ and to $0$ with all other basis elements. Thus the residue classes of $\mathcal{M}_{3(m-1)-d}$ form the exact dual basis to those of $\mathcal{M}_d$, proving that \eqref{eq:ambient-perfect-pairing} is a perfect pairing.

Applying Lemma~\ref{lem:inject} with 
$r=0$ and $Q=1\in\F_q[X_1,X_2,X_3]_0$ 
shows that $J\notin\mathfrak{q}$. Define the homogeneous colon ideal and 
the corresponding standard graded quotient algebra:
\[
 (\mathfrak{q}:J) = \{u\in\F_q[X_1,X_2,X_3]:Ju\in\mathfrak{q}\} = 
 \bigoplus_{i\ge0}(\mathfrak{q}:J)_i, \qquad
 \cG = \F_q[X_1,X_2,X_3]/(\mathfrak{q}:J) = 
 \bigoplus_{i\ge0}\cG_i,
\]
where $(\mathfrak{q}:J)_i = \{u\in\F_q[X_1,X_2,X_3]_i:Ju\in\mathfrak{q}_{q+2-h+i}\}$ and $\cG_i = \F_q[X_1,X_2,X_3]_i/(\mathfrak{q}:J)_i$. Let
$
 s=m+h-5.
$
For any $i>s$ and any $u\in\F_q[X_1,X_2,X_3]_i$, we have $\deg(Ju)=q+2-h+i >  3(m-1)$, 
thus $Ju\in\F_q[X_1,X_2,X_3]_{q+2-h+i}=\mathfrak{q}_{q+2-h+i}\subseteq\mathfrak{q}$. 
Hence $(\mathfrak{q}:J)_i=\F_q[X_1,X_2,X_3]_i$, which gives
$
 \cG_i=\{0\} 
$
for all $i>s$.
In particular, $\cG$ is finite-dimensional over $\F_q$.

For each $0\le i\le s$, write $[u]=u+(\mathfrak{q}:J)_i\in\cG_i$ for $u\in\F_q[X_1,X_2,X_3]_i$, and consider the bilinear pairing
\begin{equation}\label{eq:colon-perfect-pairing}
 \cG_i\times\cG_{s-i}\longrightarrow\F_q,
 \qquad([u],[v])\longmapsto \mathcal{L}(Juv),
\end{equation}
where $\deg(Juv)=3(m-1)$. If $u'=u+w_1$ with $w_1\in(\mathfrak{q}:J)_i$ and $v'=v+w_2$ with $w_2\in(\mathfrak{q}:J)_{s-i}$, then $Jw_1\in\mathfrak{q}$ and $Jw_2\in\mathfrak{q}$, and thus
\[
 Ju'v'-Juv = (Jw_1)v + u(Jw_2) + (Jw_1)w_2 \in \mathfrak{q}_{3(m-1)},
\]
implying $\mathcal{L}(Ju'v')=\mathcal{L}(Juv)$. 
Thus \eqref{eq:colon-perfect-pairing} is well-defined. Moreover, if $[u]\in\cG_i\setminus\{0\}$, 
then  $Ju\notin\mathfrak{q}_{q+2-h+i}$. Since $3(m-1)-(q+2-h+i)=s-i$, the perfect pairing \eqref{eq:ambient-perfect-pairing} provides a homogeneous polynomial $v\in\F_q[X_1,X_2,X_3]_{s-i}$ such that $\mathcal{L}(Juv)=\mathcal{L}((Ju)v)\ne0$, which implies $[v]\ne0$ in $\cG_{s-i}$. Interchanging $u$ and $v$ shows by symmetry that \eqref{eq:colon-perfect-pairing} is nondegenerate on both sides and hence is a perfect pairing for every $0\le i\le s$. For $i=0$, since $J\notin\mathfrak{q}$, we have $(\mathfrak{q}:J)_0=\{0\}$ and $\cG_0=\F_q$, so the perfect pairing for $i=0$ yields
$
 \dim_{\F_q}\cG_s = \dim_{\F_q}\cG_0 = 1.
$

Let $\mathfrak{m}_{\cG}=([X_1],[X_2],[X_3])=\bigoplus_{i=1}^s\cG_i$ be the homogeneous maximal ideal of $\cG$, and let $\operatorname{Soc}(\cG)=\{x\in\cG:\mathfrak{m}_{\cG}x=\{0\}\}$ be the socle of $\cG$. Since $\cG_{s+1}=\{0\}$, we have $[X_j]\cG_s\subseteq\cG_{s+1}=\{0\}$ for $1\le j\le 3$, so $\cG_s\subseteq\operatorname{Soc}(\cG)$. Conversely, suppose $0\le i<s$ and $[u]\in\cG_i\setminus\{0\}$. By the perfect pairing \eqref{eq:colon-perfect-pairing}, there exists $[v]\in\cG_{s-i}$ such that $\mathcal{L}(Juv)\ne0$, which implies $[u][v]=[uv]\ne0$ in $\cG_s$. Because $s-i\ge1$, we can decompose $v\in\F_q[X_1,X_2,X_3]_{s-i}$ as $v=\sum_{j=1}^3 X_j v_j$ with $v_j\in\F_q[X_1,X_2,X_3]_{s-i-1}$. Passing to $\cG_s$ gives
$
 0 \ne [u][v] = \sum_{j=1}^3 [v_j]\bigl([X_j][u]\bigr).
$
Consequently, $[X_j][u]\ne0$ for at least one $j\in\{1,2,3\}$,
giving $[u]\notin\operatorname{Soc}(\cG)$. 
Therefore $\operatorname{Soc}(\cG)=\cG_s$ has dimension $1$, proving that $\cG$ is a standard graded Artinian Gorenstein algebra of socle degree $s$ and embedding dimension $\dim_{\F_q}\cG_1\le\dim_{\F_q}\F_q[X_1,X_2,X_3]_1=3$.
Since $h\le m-1$, we have
$h-2\le s/2$. Applying Lemma~\ref{lem:inject} with $r=h-3\ge0$ shows that $(\mathfrak{q}:J)_{h-3}=\{0\}$. Consequently,
\[
 \dim_{\F_q}\cG_{h-3}
 =\dim_{\F_q}\F_q[X_1,X_2,X_3]_{h-3}-\dim_{\F_q}(\mathfrak{q}:J)_{h-3}
 =\binom{h-1}{2}.
\]
Applying Lemma~\ref{lem:unimodal} to $\cG$ at degree $h-2$ gives
$
 \dim_{\F_q}\cG_{h-2}\ge\dim_{\F_q}\cG_{h-3}=\binom{h-1}{2}.
$
By \eqref{eq:critical-kernel}, $\kappa=\dim_{\F_q}(\mathfrak{q}:J)_{h-2}=\dim_{\F_q}\F_q[X_1,X_2,X_3]_{h-2}-\dim_{\F_q}\cG_{h-2}$. 
Using $\dim_{\F_q}\F_q[X_1,X_2,X_3]_{h-2}=\binom{h}{2}$, we conclude that
\begin{align*}
 \kappa
 &=\dim_{\F_q}\F_q[X_1,X_2,X_3]_{h-2}-\dim_{\F_q}\cG_{h-2}\\
 &\le\binom{h}{2}-\dim_{\F_q}\cG_{h-3}\\
 &=\binom{h}{2}-\binom{h-1}{2}
 =h-1,
\end{align*}
completing the proof.
\end{proof}

Combining the bound on the orthogonal complement with the estimate for the critical kernel yields the explicit lower bound for the dimension of the sum of the local subcodes.
\begin{lemma}\label{lem:main}
For the family of subcodes
$\{\cD_{\Lambda_0\cup\{a\}}\}_{a\in\Lambda_0^c}$, we have
\begin{equation}\label{eq:canonical-span-bound}
 \dim_{\F_q}\left(\sum_{a\in\Lambda_0^c}
       \cD_{\Lambda_0\cup\{a\}}\right)\ge q-2h.
\end{equation}
\end{lemma}
\begin{proof}
Since every summand vanishes on $\Lambda_0$, coordinate projection onto $\Lambda_0^c$ is injective on the space
$\sum_{a\in\Lambda_0^c}\cD_{\Lambda_0\cup\{a\}}$ and thus preserves its dimension as a subspace of $\F_q^{\Lambda_0^c}$.
Equip $\F_q^{\Lambda_0^c}$ with the standard nondegenerate bilinear pairing 
$\langle\bb,\bz\rangle=\sum_{i\in\Lambda_0^c}b_iz_i$, and let 
$\left(\sum_{a\in\Lambda_0^c}\cD_{\Lambda_0\cup\{a\}}\right)^\perp\subseteq\F_q^{\Lambda_0^c}$ 
denote the orthogonal complement of 
$\sum_{a\in\Lambda_0^c}\cD_{\Lambda_0\cup\{a\}}$ in $\F_q^{\Lambda_0^c}$. 
Combining Propositions~\ref{prop:colon} and  \ref{prop:kappa} yields
\[
 \dim_{\F_q}\left(\sum_{a\in\Lambda_0^c}
       \cD_{\Lambda_0\cup\{a\}}\right)^\perp
 \le 3+\kappa
 \le 3+(h-1)
 = h+2.
\]
We have
\[
 \dim_{\F_q}\left(\sum_{a\in\Lambda_0^c}
       \cD_{\Lambda_0\cup\{a\}}\right)
 =(q+2-h)-\dim_{\F_q}\left(\sum_{a\in\Lambda_0^c}
       \cD_{\Lambda_0\cup\{a\}}\right)^\perp
 \ge q-2h,
\]
as required.
\end{proof}

With the canonical span bound established, we now rule out the existence of linear $[q+2,k]$ MDS codes for the intermediate parameter range. The argument forces a direct dimensional contradiction between the lower bound obtained in Lemma~\ref{lem:main} and the exact capacity of the shortened dual code.
\begin{lemma}\label{lem:no-qplus2}
Let $q$ be an even prime power and $4\le k\le q-2$.
There is no linear $[q+2,k]$ MDS code over $\F_q$.
\end{lemma}
\begin{proof}
Suppose for contradiction that a linear $[q+2,k]$ MDS code $\cC$ over $\F_q$ exists for some $4\le k\le q-2$. 
We have $q\ge 8$ and $m=q/2\ge 4$. 
The dual code $\cC^\perp$ is a linear $[q+2, q+2-k]$ MDS code. If $k > m+1$, the dimension of $\cC^\perp$ satisfies 
$4 \le q+2-k \le m$. Therefore, by replacing $\cC$ with $\cC^\perp$ whenever $k > m+1$, we may assume without loss of 
generality that the code dimension $k$ is at most $m+1$. Setting $h=k-3$, this parameter range becomes
\[
 4\le k\le m+1, \qquad 1\le h\le m-2.
\]
Fix an index subset $\Lambda_0\subset\cI$ with $|\Lambda_0|=h$. Since $\cC^\perp$ is an MDS code of dimension
whose dimension is strctly greater than $h$,
the coordinate projection $\pi_{\Lambda_0}:\cC^\perp\to\F_q^{\Lambda_0}$ given by $\bb\mapsto(b_i)_{i\in\Lambda_0}$ 
is surjective of rank $|\Lambda_0|=h$, with kernel $\ker\pi_{\Lambda_0}=\operatorname{Short}_{\Lambda_0}(\cC^\perp)$. The rank-nullity theorem yields
\[
 \dim_{\F_q}\operatorname{Short}_{\Lambda_0}(\cC^\perp)
 =\dim_{\F_q}\cC^\perp - \operatorname{rank}(\pi_{\Lambda_0})
 =q-2h-1.
\]
Combining the subspace inclusion \eqref{eq:common-shortening} with the lower bound \eqref{eq:canonical-span-bound} of Lemma~\ref{lem:main} gives
\[
 q-2h
 \le\dim_{\F_q}\left(\sum_{a\in\Lambda_0^c}
       \cD_{\Lambda_0\cup\{a\}}\right)
 \le\dim_{\F_q}\operatorname{Short}_{\Lambda_0}(\cC^\perp)
 = q-2h-1,
\]
a contradiction.
\end{proof}

With the non-existence result established for intermediate dimensions, the proof of the main theorem is completed by addressing 
the remaining boundary cases via duality and standard puncturing arguments.

\begin{proof}[Proof of Theorem \ref{thm:even}]
By Lemma~\ref{lem:mds-basic},
$2\le k\le\min\{n-2,q-1\}$. Thus no nontrivial code exists
when $q=2$, and we may assume $q\ge4$.
Puncturing an MDS code to any length $N\ge k$ preserves its
dimension and the independence of every $k$ generator columns.
If $k=2$, Lemma~\ref{lem:mds-basic} gives $n\le q+1$.
If $k=3$, the same lemma gives $n\le q+2$.

Suppose $4\le k\le q-2$. If $n\ge q+2$, puncturing to length
$q+2$ gives a linear $[q+2,k]$ MDS code, contrary to
Lemma~\ref{lem:no-qplus2}. Hence $n\le q+1$.
It remains to consider $k=q-1$. For $q=4$ this is the case
$k=3$, already treated. Suppose $q\ge8$ and $n\ge q+3$.
Puncturing to length $q+3$ and taking the dual gives an
$[q+3,4]$ MDS code. Puncturing once more gives an
$[q+2,4]$ MDS code, again contradicting
Lemma~\ref{lem:no-qplus2}. Therefore $n\le q+2$.
These cases cover all possible nontrivial dimensions.
\end{proof}

\section*{Acknowledgments}
The authors used ChatGPT to assist with exploring proof strategies,
checking intermediate calculations, and editing the exposition.
The authors are responsible for all mathematical arguments and for
the final exposition.

\end{document}